\documentclass[reqno]{amsart}

\usepackage{hyperref,amsthm,amsmath,amsfonts,amssymb,epsfig,graphicx,latexsym,float,pst-all,caption}

\newtheorem{definition}{Definition}
\newtheorem{theorem}[definition]{Theorem}
\newtheorem{remark}[definition]{Remark}
\newtheorem{assumption}[definition]{Assumption}

\newtheorem{notation}[definition]{Notation}
\newtheorem{abbreviation}[definition]{Abbreviation}

\DeclareMathOperator{\rot}{rot}
\DeclareMathOperator{\diverg}{div}

\renewcommand{\theequation}{\arabic{section}.\arabic{equation}}
\renewcommand{\thetable}{\arabic{section}.\arabic{table}}
\renewcommand{\thefigure}{\arabic{section}.\arabic{figure}}

\usepackage{dsfont,bm}

\newcommand{\kO}{o}
\newcommand{\gO}{\mathcal{O}}
\newcommand{\diffd}{\mathsf{d}}
\newcommand{\diffD}{\mathsf{D}}
\renewcommand{\vec}[1]{\bm{\mathsf{#1}}}
\newcommand{\R}{\mathds{R}}
\newcommand{\N}{\mathds{N}}
\newcommand{\E}{\mathds{E}}
\renewcommand{\S}{\mathds{S}}
\newcommand{\SO}{\mathcal{SO}}
\newcommand{\aset}[1]{\mathcal{#1}}
\newcommand{\mat}[1]{\bm{\mathsf{#1}}}
\let\Re\relax
\DeclareMathOperator{\Re}{Re}
\DeclareMathOperator{\Fr}{Fr}
\DeclareMathOperator{\diag}{diag}
\DeclareMathOperator{\tr}{tr}
\newtheorem{lem}[definition]{Lemma}
\newtheorem{cor}[definition]{Corollary}

\begin{document}

\title[Macroscopic stresses in dilute suspension of weakly inertial particles]
{Modeling of macroscopic stresses in a dilute suspension of small weakly inertial particles}

\author[A.~Vibe]{Alexander Vibe$^{1,\star}$}
\author[N.~Marheineke]{Nicole Marheineke$^{1}$}

\date{\today\\
$^\star$ \textit{Corresponding author}, email: vibe@math.fau.de, phone: +49\,9131\,85\,67215\\
$^1$ FAU Erlangen-N\"urnberg, Lehrstuhl Angewandte Mathematik I, Cauerstr.~11, D-91058 Erlangen, Germany
}

\begin{abstract}
In this paper we derive asymptotically the macroscopic bulk stress of a suspension of small inertial particles in an incompressible Newtonian fluid. We apply the general asymptotic framework to the special case of ellipsoidal particles and show the resulting modification due to inertia on the well-known particle-stresses based on the theory by Batchelor and Jeffery.
\end{abstract}

\maketitle

\noindent
{\sc Keywords.} Fluid mechanics, suspensions, bulk stress of particle suspension, stationary Stokes problem, Jeffery’s equation, small immersed rigid body, asymptotic analysis\\
{\sc AMS-Classification.} 76Axx, 76Dxx, 76Mxx,  76Txx

%%%%%%%%%%%%%%%%%%%%%%%%%%%% introduction %%%%%%%%%%%%%%%%%%%%%%%%%%%%%%%%%%%%%%
%%%%%%%%%%%%%%%%%%%%%%%%%%%%%%%%%%%%%%%%%%%%%%%%%%%%%%%%%%%%%%%%%%%%%%%%%%%%%%%%

\section{Introduction}\label{sec_0_intro}
Suspensions of small, arbitrarily shaped particles in a Newtonian fluid are of great interest in physical, biological and engineering sciences, see among others \cite{Russel2001, Dupire2012, Lindstrom2008}. Here one of the important questions, which often arises, is how the carrier flow is influenced by the suspended solids.
This effect strongly depends on the particle size and mass, the characteristics of the flow but also on the scale of interest. In general, the initial configuration of the particles in the fluid domain involves some kind of stochastic nature, additionally the particles may interact with each other and also turbulent flow fluctuations may strongly alter the deterministic path of each particle. Thus looking on a cut-out of the domain on the microscale, where each particle has a significant size and that way a strong impact on the streamlines, the problem is dominated by stochastic effects. On the macroscale on the other hand, where the size of the particles is small compared to the dimensions of the flow, the suspension can be described as a homogeneous, non-Newtonian fluid and the task consists of modeling the corresponding bulk stresses.

One step towards the derivation of the macroscopic suspension behavior is the study of the two-way coupled particle-fluid problem. Oberbeck described the disturbance flow generated by an ellipsoidal particle moving with a constant translational velocity through a stationary viscous fluid in \cite{Oberbeck1876}, Edwardes extended these results to the case of a rotating ellipsoid in \cite{Edwardes1893}. In the case of a dilute suspension of small rigid spheres in a stationary Newtonian fluid, Einstein gave a derivation of the change in viscosity of the carrier fluid by first considering the motion of a single particle in a sufficient small region, such that the corresponding flow may be treated as linear, and then solving the stationary Stokes equations for the disturbance flow. The contribution of many particles followed by superposition of every single disturbance flow field \cite{Einstein1906}. An analogous procedure was applied to ellipsoidal inertia-free particles by Jeffery \cite{Jeffery1922}. The model equation that describes the evolution of the principal axis of a prolate ellipsoid of revolution in a viscous fluid is known as Jeffery's equation. Batchelor presented a general framework of modeling stresses in a suspension of rigid particles in an incompressible Navier-Stokes flow \cite{Batchelor1970}. In the case of a dilute suspension, he applied his theory to ellipsoidal inertia-free particles using the results of Jeffery and derived an analytical term for the particle-stresses in the macroscopic description of the suspension. The study of non-dilute suspensions involves the modeling of the interactions between particles, see for example \cite{Batchelor1971, Leal1971, Dinh1981, Folgar1984}. Another general approach for suspension modeling was presented in~\cite{Zhang2010}, where Cauchy's stress principle was applied to the suspension, reproducing classical results as well as showing additional effects induced by spatial non-uniformities of the dispersed phase. The influence of particle concentration was also considered in~\cite{Prosperetti2004,Prosperetti2006}. Rigorous asymptotical homogenization strategies for suspensions were used for example in~\cite{Khruslov2004,Berezhnyi2013}. The work by Junk \& Illner \cite{Junk2007} deals with the strict asymptotic derivation of Jeffery's equation, using expansions in the small size ratio (ratio between the characteristic length scale associated to the particle and the one associated to the fluid). It yields a general strategy for constructing a correction to the undisturbed Navier-Stokes solution to account for the presence of a particle. 

In this paper we extend the asymptotic approach by Junk \& Illner by taking into account inertia. We particularly classify three different types of inertia and use the asymptotic results for a single particle in the general framework of suspension modeling according to Batchelor. We study the case of a dilute suspension and consider particle-fluid interactions as well as gravitational forces. The regard of inertia requires the introduction of a special splitting for the forces associated with the bulk stresses. This way we deduce an extended influence of the solid phase on the macroscopic behavior of the suspension induced by small deviations of the particle movement from the streamlines of the carrier fluid (Theorem~\ref{thm_kinMod_centralthm}). Apart from a modification of the particle-stress tensor we obtain an additional effective particle force term. We demonstrate the inertia related effects for the well-studied example of a dilute suspension of (prolate) ellipsoidal particles (Corollary~\ref{lem_specCase_ellipsoidalSusp}).

We organize this paper as follows:
In Section~\ref{sec_1_mathMod} we set up the model for a single particle in a fluid flow with respect to different types of inertia. We give a brief overview of the asymptotic derivation including the essential definitions and assumptions. In Section~\ref{sec_3_kinMod} we sketch the basic ideas of Batchelor's framework and use the asymptotic results to formulate corresponding consistent suspension models which account for weakly inertial particles. We illustrate the results for the special case of ellipsoidal particles in Section~\ref{sec_4_specCase} and conclude with a summary in Section~\ref{sec_5_conc}. Technical details to the example are provided in the Appendix. 

Throughout this paper we use the following notation:
\begin{notation}\label{not_intro_vec}
Scalar values are typeset as ordinary characters $a\in\R$. Vectors are indicated by small bold characters, $\vec{v}\in\R^n$, $n>1$, their components are denoted by $(\vec{v})_i=v_i\in\R$, $i=1,\ldots,n$. In the following we will omit the ranges of the indices if they are obvious. Matrices are written as large bold letters, $\mat{M}\in\R^{n\times m}$ with the components $(\mat{M})_{ij}=M_{ij}$, the identity matrix is denoted by $\mat{I}$. Tensors of higher order are treated as linear mappings between the corresponding spaces and written as ordinary large characters, e.g. $B:\R^3\to\R^{3\times3}$, or as the corresponding dyadic product.
 
We use a tensor calculus notation, viz.
$\vec{v}\cdot\vec{w}=\sum_{i}v_iw_i\in\R$ for all vectors $\vec{w},\vec{v}$, resp. $(\mat{A}\cdot\vec{v})_i=\sum_{j}A_{ij}v_j$ for $\mat{A}\in\R^{n\times m}$, $\vec{v}\in\R^{m}$, and $(\mat{A}\cdot\mat{B})_{ik}=\sum_{j}A_{ij}B_{jk}$ for $\mat{A}\in\R^{n\times m},\mat{B}\in\R^{m\times\ell}$. The dyadic product $\vec{v}\otimes\vec{w}\in\R^{n\times m}$ of two vectors $\vec{v}\in\R^n,\vec{w}\in\R^m$ involves $(\vec{v}\otimes\vec{w})\cdot\vec{u}=(\vec{w}\cdot\vec{u})\vec{v}$ for all $\vec{u}\in\R^{m}$, analogously for higher order tensors: $\vec{u}_1\otimes\vec{u}_2\otimes\ldots\otimes\vec{u}_n$. A two-dot product is denoted by $(\vec{v}\otimes\vec{w}):(\vec{x}\otimes\vec{y})=(\vec{v}\cdot\vec{x})(\vec{w}\cdot\vec{y})$ for $\vec{v},\vec{x}\in\R^{n}$, $\vec{w},\vec{y}\in\R^m$, analogously $(\vec{u}\otimes\vec{v}\otimes\vec{w}):(\vec{x}\otimes\vec{y})=(\vec{v}\cdot\vec{x})(\vec{w}\cdot\vec{y})\vec{u}$ and so on.
For $\vec{v},\vec{w}\in\R^3$ the cross-product is given by $(\vec{v}\times\vec{w})_k=\sum_{i,j}\epsilon_{ijk}v_iw_j$, with the Levi-Civita symbol
\begin{align*}
\epsilon_{ijk}=\begin{cases}
                 1,&\text{for $(i,j,k)$ an even permutation of $(1,2,3)$},\\
                 -1,&\text{for $(i,j,k)$ an odd permutation of $(1,2,3)$},\\
                 0,&\text{else}.
                \end{cases}
\end{align*}
We use the mapping $B(\vec{v})_{ij}=\sum_k\epsilon_{ikj}v_k$ which assigns a vector to a skew-symmetric matrix such that $B(\vec{v})\cdot\vec{x}=\vec{v}\times\vec{x}$ for all $\vec{x},\vec{v}\in\R^3$.

As for the differential operators, the Jacobian is denoted by $(\partial_{\vec{y}}\vec{f})_{ij}=\partial_{y_j}f_i$ and the gradient by its transpose $(\nabla_{\vec{y}}\vec{f})_{ij}=\partial_{y_i}f_j$. For scalar-valued functions these notations are certainly interchangeable. As usual, the nabla operator is formally treated as a vector and for example used for the divergence and rotation operators: $\diverg(\vec{u})=\nabla\cdot\vec{u}$, $\rot(\vec{u})=\nabla\times\vec{u}$. Higher order derivatives are noted as $\partial_{\vec{y}\ldots\vec{y}}\vec{u}$ for a sufficiently smooth $\vec{u}$.
\end{notation}

%%%%%%%%%%%%%%%%%%%%%%%%%%%% section 1 %%%%%%%%%%%%%%%%%%%%%%%%%%%%%%%%%%%%%%%%%
%%%%%%%%%%%%%%%%%%%%%%%%%%%%%%%%%%%%%%%%%%%%%%%%%%%%%%%%%%%%%%%%%%%%%%%%%%%%%%%%

\setcounter{equation}{0} \setcounter{figure}{0}
\section{Mathematical model of a single particle in a fluid}\label{sec_1_mathMod}
In this section we provide the ingredients for the suspension model. For this purpose, we consider a single small oriented particle suspended in a Newtonian fluid and study its influence on the carrier flow under the assumption that the particle is small compared with the dimensions of the flow. Proceeding from the three-dimensional interface problem we derive an asymptotic model for the particle in the fluid flow by help of expansions in the size parameter $\epsilon$ (size ratio between particle and flow domain). The asymptotic analysis follows the procedure that was established for inertia-free particles in \cite{Junk2007}. The novelty is the extension to inertial particles which we classify with respect to different types.

\subsection{Three-dimensional interface problem}\label{subsec_mathMod_threedimInterface}

\subsubsection*{Dimensional problem}
A particle is a rigid body that we model as an open, bounded and time-dependent domain $\aset{E}(t)\subset\R^3$ with a smooth boundary $\partial\aset{E}(t)$ and time $t\in\R^+_0$. The particle is characterized by its center of mass $\vec{c}:\R_0^+\to\R^3$, $\vec{c}(t)=|\aset{E}(t)|^{-1}\int_{\aset{E}(t)}\vec{x}\diffd\vec{x}$ and an orthonormal right-handed director triad $\vec{d}_i:\R_0^+\to\S^2_1$, $i=1,2,3$, where $\S^2_1$ denotes the unit sphere in $\R^3$. Additionally we define a time-invariant reference state $\aset{E}\subset\R^3$ with the center of mass lying in the origin and the directors identified with the Cartesian basis vectors $\{\vec{e}_1,\vec{e}_2,\vec{e}_3\}$. The function which maps every point of the reference state to the actual time-dependent state is completely presented by rigid body motions (translation and rotation), i.e.\ $\vec{x}:\aset{E}\times\R_0^+\to\aset{E}(t)$,  $\vec{x}(\vec{y},t)=\mat{R}(t)\cdot\vec{y}+\vec{c}(t)$ (see Figure~\ref{fig_mathMod_path}), with $\mat{R}\in\SO(3)$ being the rotation matrix associated with the director triad, $R_{ij}=\vec{e}_i\cdot\vec{d}_j$.
\begin{figure}
\centering
\scalebox{1}{
\psset{unit=1cm}
 \begin{pspicture}(0,0)(14,5)
  %\psgrid[subgriddiv=4, griddots=10, gridlabels=7pt](0,0)(14,5)
  %%%%%%%%%%%%%%%%% reference coordinates
  \psline[linewidth=1pt]{->}(3.5,1.5)(5.5,1.5)
  \psline[linewidth=1pt]{->}(3.5,1.5)(3.5,3.5)
  \rput(3.25,3.5){$\vec{e}_2$}
  \rput(5.5,1.25){$\vec{e}_1$}
  \rput(3.375,1.375){$\vec{0}$}
  %%%%%%%%%%%%%%%%% 
  %%%%%%%%%%%%%%%%% cartesian coordinates new state
  \psline[linewidth=1pt]{->}(8.0,1.5)(10.0,1.5)
  \psline[linewidth=1pt]{->}(8.0,1.5)(8.0,3.5)
  \rput(7.75,3.5){$\vec{e}_2$}
  \rput(10,1.25){$\vec{e}_1$}
  \rput(7.875,1.375){$\vec{0}$}
  %%%%%%%%%%%%%%%%% %%%%%%%%%%%%%%%%% %%%%%%%%%%%%%%%%% 
  %%%%%%%%%%%%%%%%% first object
  \psbezier[showpoints=false,linewidth=1pt]{-}(2.5,1.5)(2,2.5)(2,3.5)(3.25,3.0)
  \psbezier[showpoints=false,linewidth=1pt]{-}(3.25,3.0)(4.25,2.5)(4.75,1.5)(4.5,1.25)
  \psbezier[showpoints=false,linewidth=1pt]{-}(4.5,1.25)(3.75,.5)(3.25,.5)(2.5,1.5)
  \rput(2.75,2.75){$\aset{E}$}
  %%%%%%%%%%%%%%%%% 
  %%%%%%%%%%%%%%%%% center of mass
  \psline[linewidth=1pt,linestyle=dashed]{->}(8,1.5)(11,2.5)
  \rput(10.625,2.75){$\vec{c}(t)$}
  %%%%%%%%%%%%%%%%% director triad
  \psline[linewidth=1pt]{->}(11,2.5)(12.4142,1.0858)
  \psline[linewidth=1pt]{->}(11,2.5)(12.4142,3.9142)
  \rput(11.875,1){$\vec{d}_1(t)$}
  \rput(12.5,3.5){$\vec{d}_2(t)$}
  %%%%%%%%%%%%%%%%% second object
  \psbezier[showpoints=false,linewidth=1pt]{-}(10.2929,3.2071)(10.6464,4.2678)(11.3536,4.9749)(11.8839,3.7374)
  \psbezier[showpoints=false,linewidth=1pt]{-}(11.8839,3.7374)(12.2374,2.6768)(11.8839,1.6161)(11.5303,1.6161)
  \psbezier[showpoints=false,linewidth=1pt]{-}(11.5303,1.6161)(10.4697,1.6161)(10.1161,1.9697)(10.2929,3.2071)
  \rput(11.25,3.8){$\aset{E}(t)$}
  %%%%%%%%%%%%%%%%% transition
  \psbezier[showpoints=false,linewidth=1pt,linestyle=dotted]{<-}(4.5,.75)(5,0)(8.5,0)(9,.75)
  \psbezier[showpoints=false,linewidth=1pt,linestyle=dotted]{->}(4.5,4.25)(5,5)(8.5,5)(9,4.25)
  \rput(6.75,4.5){$\vec{x}(\vec{y},t)$}
  \rput(6.75,0.5){$\vec{y}(\vec{x},t)$}
 \end{pspicture}
}
 \caption{Lagrangian description, bijective mapping between reference state and the actual time-dependent state.}\label{fig_mathMod_path}
\end{figure}
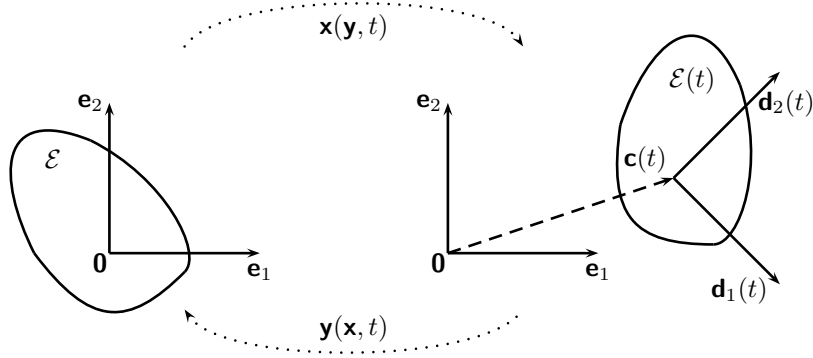
We assume that the small oriented particle is suspended in an incompressible Newtonian fluid. The regular domain $\Omega$ contains the fluid and the particle. The acting forces are due to particle-fluid interaction and gravity. The model of first principles for the particle and the flow consists for all $t>0$ of the incompressible Navier-Stokes equations in the fluid domain $\Omega\setminus\overline{\aset{E}}(t)$, the no-slip condition on the interface as well as the kinematics and dynamics of the particle:
\begin{subequations}\label{eq_mathMod_full_dimensional}
\begin{align}\label{eq_mathMod_NSE_dimensional}
\hspace*{-0.3cm} \rho_f\left(\partial_t\vec{u}(\vec{x},t)+(\vec{u}(\vec{x},t)\cdot\nabla)\vec{u}(\vec{x},t)\right)&=\nabla\cdot\mat{S}[\vec{u}]^T(\vec{x},t)+\rho_fg\vec{e}_g,\qquad \nabla\cdot\vec{u}=0,\quad  \vec{x}\in\Omega\setminus\overline{\aset{E}}(t),\\
 \vec{u}(\vec{x},t)&=\vec{\omega}(t)\times\left(\vec{x}-\vec{c}(t)\right)+\vec{v}(t), \quad \hspace*{2.1cm} \vec{x}\in\partial\aset{E}(t),\\\label{eq_mathMod_R_dimensional}
 \frac{\diffd}{\diffd t}\vec{c}(t)=\vec{v}(t),&\qquad \frac{\diffd}{\diffd t}\mat{R}(t)=B(\vec{\omega}(t))\cdot\mat{R}(t),
 \end{align}
 \begin{align}
 \rho_p|\aset{E}|\frac{\diffd}{\diffd t}\vec{v}(t)&=\int_{\partial\aset{E}}\mat{S}[\vec{u}](\vec{x}(\vec{y},t),t)\cdot\mat{R}(t)\cdot\vec{n}(\vec{y})\,\diffd s(\vec{y})+\rho_p|\aset{E}|g\vec{e}_g,\\
 \rho_p|\aset{E}|\frac{\diffd}{\diffd t}\left(\mat{J}(t)\cdot\vec{\omega}(t)\right)&=\mat{R}(t)\cdot\int_{\partial\aset{E}}\vec{y}\times\left(\mat{R}(t)^T\cdot\mat{S}[\vec{u}](\vec{x}(\vec{y},t),t)\cdot\mat{R}(t)\cdot\vec{n}(\vec{y})\right)\,\diffd s(\vec{y}).
 \end{align}
\end{subequations}
Here, $\vec{u}$ and $p$ denote the fluid velocity and pressure, $\mat{S}[\vec{u}]=-p\mat{I}+\mu_f\left(\nabla\vec{u}+\nabla\vec{u}^T\right)$ the Newtonian stress tensor with $\mu_f$ being the fluid viscosity and $\rho_f$ the fluid density. The particle related quantities $\vec{v}$ and $\vec{\omega}$ describe the linear and angular velocities of the rigid body. The inertia tensor is given by $\mat{J}(t)=\mat{R}(t)\cdot\hat{\mat{J}}\cdot\mat{R}(t)^T$ with the time-invariant part $\hat{\mat{J}}=|\aset{E}|^{-1}\int_{\aset{E}}\|\vec{y}\|^2\mat{I}-\vec{y}\otimes\vec{y}\,\diffd\vec{y}$ and $\rho_p$ is the particle density. The gravitational acceleration is given by $g$ with normalized direction $\vec{e}_g$. The system is completed by appropriate initial conditions for the flow and the particle as well as boundary conditions for the flow on $\partial\Omega$.

\subsubsection*{Dimensionless formulation}
We introduce two characteristic lengths associated with the fluid system $\overline{x}$ and with the particle $\overline{y}$ as well as a characteristic time $\overline{t}$. We scale the velocities as $\overline{u}=\overline{v}=\overline{x}/\overline{t}$ and the pressure with $\overline{p}=\overline{u}\mu_f/\overline{x}$. For the sake of a simple notation, we keep the same symbols for all variables as in the dimensional formulation and get the dimensionless system:
\begin{subequations}\label{eq_mathMod_dimFreeSystem}
\begin{align}
 \Re\left(\partial_t\vec{u}+(\vec{u}\cdot\nabla)\vec{u}\right)&=\nabla\cdot\mat{S}[\vec{u}]^T+\Re\Fr^{-2}\vec{e}_g,\qquad \nabla\cdot\vec{u}=0,\quad \vec{x}\in\Omega\setminus\overline{\aset{E}}(t),\label{eq_mathMod_navSt}\\
 \vec{u}&=\vec{\omega}\times\left(\vec{x}-\vec{c}\right)+\vec{v}, \quad \hspace*{21.7ex} \vec{x}\in\partial\aset{E}(t),\label{eq_mathMod_noslip}\\
 \frac{\diffd}{\diffd t}\vec{c}=\vec{v},&\qquad \frac{\diffd}{\diffd t}\mat{R}=B(\vec{\omega})\cdot\mat{R},\label{eq_mathMod_kin}\\
 \rho\epsilon\Re|\aset{E}|\frac{\diffd}{\diffd t}\vec{v}&=\int_{\partial\aset{E}}\mat{S}[\vec{u}]\cdot\mat{R}\cdot\vec{n}\,\diffd s+\rho\epsilon\Re\Fr^{-2}|\aset{E}|\vec{e}_g,\label{eq_mathMod_linMom}\\
 \rho\epsilon^2\Re|\aset{E}|\frac{\diffd}{\diffd t}\left(\mat{J}\cdot\vec{\omega}\right)&=\mat{R}\cdot\int_{\partial\aset{E}}\vec{y}\times\left(\mat{R}^T\cdot\mat{S}[\vec{u}]\cdot\mat{R}\cdot\vec{n}\right)\,\diffd s.\label{eq_mathMod_angMom}
\end{align}
\end{subequations}
The model~\eqref{eq_mathMod_dimFreeSystem} is characterized by four dimensionless parameters: the size ratio $\epsilon=\overline{y}/\overline{x}$, the density ratio $\rho=\rho_p/\rho_f$, the Reynolds number $\Re=\overline{u}\,\overline{x}\rho_f/\mu_f$ (ratio of inertial and viscous forces in the fluid) and the Froude number $\Fr=\overline{u}/\sqrt{\overline{x}g}$ (ratio of inertial and gravitational forces in the fluid). The Newtonian stress in dimensionless form\footnote{In the following $\mat{S}[.]$ always designates the Newtonian stress tensor of the corresponding solutions of the Navier-Stokes or the Stokes equations and $\mat{E}[.]$ denotes the associated deformation gradient tensor.} reads as $\mat{S}[\vec{u}]=-p\mat{I}+2\mat{E}[\vec{u}]$ with $\mat{E}[\vec{u}]=0.5(\nabla\vec{u}+\nabla\vec{u}^T)$. The size parameter enters the bijection between the reference and time-dependent state:
\begin{subequations}\label{eq_mathMod_trafoBoth}
\begin{align}
 \vec{x}(\vec{y},t,\epsilon)&=\epsilon\mat{R}(t)\cdot\vec{y}+\vec{c}(t),\label{eq_mathMod_trafo}\\
 \vec{y}(\vec{x},t,\epsilon)&=\epsilon^{-1}\mat{R}^T(t)\cdot\left(\vec{x}-\vec{c}(t)\right).\label{eq_mathMod_backTrafo}
\end{align}
\end{subequations}
The density ratio $\rho$ is now the key parameter that allows us to define different inertial regimes. 

\subsection{Tracer particles of different inertial type}\label{subsec_mathMod_inertia}

As we will show our asymptotic approach results in particles which follow the streamlines of the surrounding fluid. Hence, we refer to them as tracer particles. To distinguish between different types of inertial particles, we introduce an $\epsilon$-dependent mass function $\alpha_{mass}(\epsilon)$. With the particle density $\rho_p=\rho_p(\epsilon)=\alpha_{mass}(\epsilon)m_p/V_p$, the density ratio and the size parameter are then related according to
\begin{align}
 \rho=\frac{\rho_p}{\rho_f}=\alpha_{mass}(\epsilon)\frac{m_pV_f}{m_fV_p}=\alpha_{mass}(\epsilon)\frac{m_p\overline{x}^3}{m_f|\aset{E}|\overline{y}^3}=\frac{\alpha_{mass}(\epsilon)}{\epsilon^3}\frac{m}{|\aset{E}|}.\label{eq_mathMod_massfcn}
\end{align}
Here $m_p,V_p$ and $m_f,V_f$ are the mass and the volume associated with the particle and the fluid, respectively, $m=m_p/m_f$. The freely selectable mass function has no physical meaning, but allows to balance the inertial terms in the particle momentum balance in different ways, yielding several models of inertial particles. In particular, we set $\alpha_{mass}(\epsilon)=\epsilon^k$, $k\in\N$, and define three inertial regimes with respect to the behavior of the density ratio in the limit $\epsilon\to0$, i.e.\ heavy ($k=2$), normal ($k=3$) and light weighted ($k>3$) tracer particles, see Table~\ref{tab_mathMod_inertReg}. For $k=3$, the density ratio satisfies $\rho\equiv const$ leading to the inertia-free particle model that was investigated in \cite{Junk2007}. The cases $k\leq1$ are not covered by the asymptotic simplification as we will comment on in Remark~\ref{rem_asymAn_impactOfScaling}.
\begin{table}
\begin{tabular}{|c|l|l|}
\hline
 $\alpha_{mass}=\epsilon^k$ & Name of type & Behavior of density ratio\\\hline
 $k>3$&light weighted tracer particles&$\rho\to0$\\
 $k=3$&normal tracer particles&$\rho\equiv const$\\
 $k=2$&heavy tracer particles&$\rho\to\infty$\\\hline
\end{tabular}
\caption{Classification of inertial types by means of the mass function, cf.\ \eqref{eq_mathMod_massfcn}.}\label{tab_mathMod_inertReg}
\end{table}

\subsection{Asymptotic analysis}\label{sec_2_asymAn}

In this subsection we generalize the asymptotic approach of \cite{Junk2007} to the introduced inertial type classification. We particularly point out the differences in the asymptotic analysis that arise for the heavy tracer particles.

\subsubsection*{Asymptotic expansion}
We assume that under the considered inertia types the particle influences the carrier flow locally but not globally. This way we follow \cite{Junk2007} and make the following assumptions.
\begin{assumption}[Particle-fluid interaction]\label{assum_mathMod_assumptions1}
\hfill
 \begin{itemize}
 \item[1)] The fluid is essentially undisturbed by the particle.
 \item[2)] The fluid motion induces a rotation of the particle of $\gO(1)$.
 \item[3)] The particle is far away from the boundary of the fluid domain $\partial\Omega$.
\end{itemize}
\end{assumption}\noindent
For the quantities of the particle we take a regular expansion in powers of $\epsilon$, while the fields of the fluid flow are separated in a global and a local part: 
\begin{align*}
 \vec{a}=\sum_{i=0}^{\infty}\epsilon^i\vec{a}_i\quad \text{for } \vec{a}\in\{\vec{c},\vec{v},\vec{\omega},\mat{R}\},\qquad \vec{u}=\vec{u}_0+\mat{R}\cdot\vec{u}_{loc},\qquad p=p_0+p_{loc}. 
\end{align*}
The local fields of the fluid quantities are expressed in the local coordinates of the particle reference state,
\begin{align*}
 \vec{u}_{loc}(\vec{x},t,\epsilon)=\sum_{i=1}^{\infty}\epsilon^i\vec{u}_{loc,i}(\vec{y}(\vec{x},t,\epsilon),t),\qquad p_{loc}(\vec{x},t,\epsilon)=\sum_{i=1}^{\infty}\epsilon^{i-1}p_{loc,i}(\vec{y}(\vec{x},t,\epsilon),t).
\end{align*}
The scaling is chosen in such a way that the gradient of the pressure balances the Laplacian of the velocity. It is a consequence of the fact that the bijective mapping between the reference and time-dependent state is $\epsilon$-dependent, see~\eqref{eq_mathMod_backTrafo}. (This way the chain rule generates a factor of $\epsilon^{-1}$ every time $\vec{u}_{loc}$ is being differentiated with respect to $\vec{x}$ or $t$.)

\subsubsection*{Asymptotic solution}
Our goal is to formulate a solution of the complete system \eqref{eq_mathMod_dimFreeSystem} up to an error of $\gO(\epsilon)$. Therefore, we consider only a finite number of asymptotic coefficients in the expansions. To address them we set
\begin{subequations}\label{eq_asymAn_expanAll}
 \begin{align}
 &\vec{a}^e=\sum_{i=0}^2\epsilon^i\vec{a}_i\quad \text{for } \vec{a}\in\{\vec{c},\vec{v}\},\qquad \vec{a}^e=\sum_{i=0}^1\epsilon^i\vec{a}_i\quad  \text{for }  \vec{a}\in\{\vec{\omega},\mat{R}\},\label{eq_asymAn_expanOth}\\
 &\vec{u}^e=\vec{u}_0+\mat{R}^e\sum_{i=1}^2\epsilon^{i}\vec{u}_{loc,i},\qquad\hspace*{4.2ex} p^e=p_0+\sum_{i=1}^2\epsilon^{i-1}p_{loc,i}.\label{eq_asymAn_expanFlow}
\end{align}
\end{subequations}
Inserting the expansion coefficients into~\eqref{eq_mathMod_trafoBoth} yields the following form of the bijective mapping between the reference and the time-dependent state:
\begin{align*}
 \vec{x}^e(\vec{y},t,\epsilon)&=\epsilon\mat{R}^e(t,\epsilon)\cdot\vec{y}+\vec{c}^e(t,\epsilon),\qquad \vec{y}^e(\vec{x},t,\epsilon)=\epsilon^{-1}\left(\mat{R}^e(t,\epsilon)\right)^{-1}\cdot(\vec{x}-\vec{c}^e(t,\epsilon)).
\end{align*}
It is shown in (\cite{Junk2007}, Lemma~5) that $\left(\mat{R}^e\right)^{-1}=\left(\mat{R}^e\right)^T+\gO(\epsilon^2)$ and $\|\left(\mat{R}^e\right)^{-1}\|=\gO(1)$, thus $\vec{y}^e$ is well-defined. The asymptotic coefficients of the mapping are particularly given by
\begin{align*}
\vec{x}^e(\vec{y},t,\epsilon)=\sum_{i=0}^2\epsilon^i\vec{x}_i(\vec{y},t),\qquad \text{with } \vec{x}_0=\vec{c}_0,\quad \vec{x}_1=\mat{R}_0\cdot\vec{y}+\vec{c}_1,\quad \vec{x}_2=\mat{R}_1\cdot\vec{y}+\vec{c}_2.
\end{align*}
Moreover, we use Taylor expansions for functions of the form $f(\vec{x}(\vec{y},t,\epsilon))$ in $\vec{c}_0$ in the following, i.e.\
\begin{align*}
 f(\vec{x})&=f(\vec{c}_0)+\partial_{\vec{x}}f(\vec{c}_0)\cdot\left(\epsilon\vec{x}_1+\epsilon^2\vec{x}_2\right)+0.5\partial_{\vec{x}\vec{x}}f(\vec{c}_0):\left(\epsilon\vec{x}_1+\epsilon^2\vec{x}_2\right)\otimes\left(\epsilon\vec{x}_1+\epsilon^2\vec{x}_2\right)+\gO(\epsilon^3)\\
 &=\sum_{i=0}^2\epsilon^i\diffD_if(\vec{c}_0)+\gO(\epsilon^3),
\end{align*}
with the differential operators
\begin{align*}
 \diffD_0=1,\qquad \diffD_1=\vec{x}_1\cdot\nabla,\qquad \diffD_2=\vec{x}_2\cdot\nabla+0.5\vec{x}_1\otimes\vec{x}_1:\nabla^2.
\end{align*}

To emphasize the general structure of the asymptotic result in the subsequent Lemma~\ref{lem_asymAna_junk} and to stress the relevance of certain terms for the suspension model in Section~\ref{sec_3_kinMod} we introduce abbreviations for some expressions.

\begin{abbreviation}\label{abb_asymAn_1} Consider the following model-relevant functions for $\alpha_{mass}(\epsilon)=\epsilon^k$, $k\geq 2$
\begin{align*}
  \vec{h}_1&=\mat{R}_0^T\cdot\left(\vec{v}_1+B(\vec{\omega}_0)\cdot\mat{R}_0\cdot\vec{y}-\diffD_1\vec{u}_0\right),\\
 \vec{h}_2&=\mat{R}_0^T\cdot\left(\vec{v}_2+\left(B(\vec{\omega}_0)\cdot\mat{R}_1+B(\vec{\omega}_1)\cdot\mat{R}_0\right)\cdot\vec{y}-\diffD_2\vec{u}_0-\mat{R}_1\cdot\vec{u}_{loc,1}\right),\\
  \vec{f}_1&=\begin{cases} %eq_asymAn_rightHand
                  \vec{0},&k\geq3\\
                  \vec{k}_0,&k=2
                 \end{cases},\\
 \vec{f}_2&=-\mat{R}_0^T\cdot\mat{R}_1\cdot\vec{f}_1-|\aset{E}|\mat{R}_0^T\cdot \nabla_{\vec{x}}\cdot\mat{S}[\vec{u}_0]^T+
		\begin{cases}
                 \vec{0},&k\geq4\\
                 \vec{k}_0,&k=3\\
                 \vec{k}_1,&k=2
                \end{cases},\\
 \vec{g}_1&=\vec{0},\quad k\geq2,\\
 \vec{g}_2&=-\mat{R}_0^T\cdot\mat{R}_1\cdot\vec{g}_1-\int_{\partial\aset{E}}\vec{y}\times\left(\mat{R}_0^T\cdot \diffD_1\mat{S}[\vec{u}_0]\cdot\mat{R}_0\cdot\vec{n}\right)\,\diffd s
		+\begin{cases}
		  \vec{0},&k\geq3\\
		  \vec{\ell}_0,&k=2
		 \end{cases},
\end{align*}
where the derivatives of $\vec{u}_0$ and the Newtonian stresses are evaluated in $\vec{c}_0$. Moreover, abbreviate the linear accelerations by
\begin{align*}
 \vec{k}_0=m\Re\mat{R}_0^T\cdot\left(\frac{\diffd}{\diffd t}\vec{v}_0-\frac{1}{\Fr^2}\vec{e}_g\right),\qquad\vec{k}_1=m\Re\mat{R}_0^T\cdot\frac{\diffd}{\diffd t}\vec{v}_1,
\end{align*}
and the angular accelerations by
\begin{align*}
 \vec{\ell}_0&=m\Re\mat{R}_0^T\cdot\frac{\diffd}{\diffd t}\left(\mat{R}_0\cdot\hat{\mat{J}}\cdot\mat{R}_0^T\cdot\vec{\omega}_0\right),\\
 \vec{\ell}_1&=m\Re\mat{R}_0^T\cdot\frac{\diffd}{\diffd t}\left(\mat{R}_0\cdot\hat{\mat{J}}\cdot\mat{R}_0^T\cdot\vec{\omega}_1+\left(\mat{R}_0\cdot\hat{\mat{J}}\cdot\mat{R}_1^T+\mat{R}_1\cdot\hat{\mat{J}}\cdot\mat{R}_0^T\right)\cdot\vec{\omega}_0\right).
 \end{align*}
\end{abbreviation}

\begin{lem}[Asymptotic one-particle model]\label{lem_asymAna_junk}
Let the following four requirements be fulfilled, then $\vec{u}^e$, $p^e$, $\vec{c}^e$, $\vec{v}^e$, $\vec{\omega}^e$ and $\mat{R}^e$ defined in~\eqref{eq_asymAn_expanAll} are a solution of the complete system \eqref{eq_mathMod_dimFreeSystem} up to an order of $\gO(\epsilon)$ for $\epsilon\downarrow0$.
\begin{itemize}
 \item[(R1)] Let the global flow velocity $\vec{u}_0$ and pressure $p_0$ be the solutions of the incompressible Navier-Stokes equations in $\Omega\times\R^+$
 \begin{align}
  \Re\left(\partial_t\vec{u}_0+(\vec{u}_0\cdot\nabla_{\vec{x}})\vec{u}_0\right)=\nabla_{\vec{x}}\cdot\mat{S}[\vec{u}_0]^T+\Re\Fr^{-2}\vec{e}_g,\qquad \nabla_{\vec{x}}\cdot\vec{u}_0=0\label{eq_asymAn_navStokesMomentum}
 \end{align}
with the boundary condition $\vec{u}_0=\vec{u}$ on $\partial\Omega$, and the initial condition $\vec{u}_0(\vec{x},0)=\vec{u}(\vec{x},0)$ for $\vec{x}\in\Omega\setminus\aset{E}(0)$, and $\vec{u}_0(\vec{x},0)=\vec{\omega}(0)\times(\vec{x}-\vec{c}(0))+\vec{v}(0)$ for $\vec{x}\in\aset{E}(0)$. 
\item[(R2)] Let the particle related coefficients satisfy the following set of conditions: Let
\begin{subequations}\label{eq_asymAn_particleEqs}
\begin{align}
 \frac{\diffd}{\diffd t}\vec{c}_i(t)&=\vec{v}_i(t),\quad i=0,1,2\label{eq_asymAn_partVi}
\end{align}
with the initial conditions $\vec{c}_0(0)=\vec{c}(0)$, $\vec{c}_i(0)=\vec{0}$, $i=1,2$ and let the zero-order velocity suffice the so-called `tracer condition'
\begin{align}
 \vec{v}_0(t)=\vec{u}_0(\vec{c}_0(t),t).\label{eq_asymAn_tracerCond}
\end{align}
Let additionally the matrices $\mat{R}_0$, $\mat{R}_1$ solve the differential equations
\begin{align}
 \frac{\diffd}{\diffd t}\mat{R}_0=B(\vec{\omega}_0)\cdot\mat{R}_0,\qquad\frac{\diffd}{\diffd t}\mat{R}_1=B(\vec{\omega}_1)\cdot\mat{R}_0+B(\vec{\omega}_0)\cdot\mat{R}_1\label{eq_asymAn_SOode}
\end{align}
\end{subequations}
with the initial conditions $\mat{R}_0(0)=\mat{R}(0)$ and $\mat{R}_1(0)=\mat{0}$.
\item[(R3)] For every $t>0$ let the local fields solve the stationary Stokes equations on the unbounded exterior of $\aset{E}$
\begin{subequations}\label{eq_asymAn_statStokesComplete}
\begin{align}
 \nabla_{\vec{y}}p_{loc,i}(\vec{y},t)=\Delta_{\vec{y}}\vec{u}_{loc,i}(\vec{y},t),\qquad \nabla_{\vec{y}}\cdot\vec{u}_{loc,i}=0\label{eq_asymAn_statStokesBasic}
\end{align}
with the Dirichlet and integral conditions
\begin{align}
 \vec{u}_{loc,i}(\vec{y},t)&=\vec{h}_i(\vec{y},t),\quad\vec{y}\in\partial\aset{E},\label{eq_asymAn_dirichlet}\\
 \int_{\partial\aset{E}}\mat{S}[\vec{u}_{loc,i}]\cdot\vec{n}\,\diffd s(\vec{y})&=\vec{f}_i,\label{eq_asymAn_force}\\
 \int_{\partial\aset{E}}\vec{y}\times\left(\mat{S}[\vec{u}_{loc,i}]\cdot\vec{n}\right)\,\diffd s(\vec{y})&=\vec{g}_i,\label{eq_asymAn_couple}
\end{align}
for $i=1,2$ and the functions $\vec{h}_i$, $\vec{f}_i$, $\vec{g}_i$ given in Abbreviation~\ref{abb_asymAn_1}. Additionally let $\vec{u}_{loc,i}$ and $p_{loc,i}$ fulfill the following decay properties:
\begin{align}
 \|\vec{u}_{loc,i}\|\leq c_i/\|\vec{y}\|,\qquad \|\nabla_{\vec{y}}\vec{u}_{loc,i}\|,|p_{loc,i}|\leq d_i/\|\vec{y}\|^2,\qquad\text{for } \|\vec{y}\|\geq f_i>0\label{eq_asymAn_decay}
\end{align}
\end{subequations}
for some $c_i$, $d_i$, $f_i$, $i=1,2$ independent of $\vec{y}$.
\item[(R4)] Let the mass function \eqref{eq_mathMod_massfcn} be given as $\alpha_{mass}(\epsilon)=\epsilon^k$ with $k\geq2$.
\end{itemize}
\end{lem}

\begin{proof}
The proof results from a straight forward computation where the asymptotic coefficients are inserted into the full system~\eqref{eq_mathMod_dimFreeSystem} and Taylor expansions are applied for the global flow fields, whenever they are evaluated at the particle boundary (see \cite{Junk2007} for normal tracer particles). The introduction of the different inertial types does not change the general mathematical structure of the problem but only affects the linear and angular momentum balances of the particle, i.e.\ the functions $\vec{f}_i$, $\vec{g}_i$ in (R3). We sketch the steps of the proof for the sake of completeness.

Inserting the asymptotic expansions for $\vec{c}$, $\vec{v}$, $\vec{\omega}$ and $\vec{R}$ in \eqref{eq_mathMod_kin} and applying \eqref{eq_asymAn_partVi}, \eqref{eq_asymAn_SOode} shows that the particle kinematics \eqref{eq_mathMod_kin} is fulfilled up to an order of $\gO(\epsilon^2)$.

The conservation of mass in \eqref{eq_mathMod_navSt} is fulfilled exactly since each local velocity field is assumed to be divergence-free.

Considering the no-slip equation~\eqref{eq_mathMod_noslip},  the use of the asymptotic expansions and the Taylor series of $\vec{u}_0$ in $\vec{c}_0$ results in the following conditions that are exactly the Dirichlet conditions of $\vec{u}_{loc,i}$~\eqref{eq_asymAn_dirichlet} and the tracer condition~\eqref{eq_asymAn_tracerCond},
\begin{align*}
  \vec{u}_0(\vec{c}_0(t),t)&=\vec{v}_0,\\
  \diffD_1\vec{u}_0(\vec{c}_0(t),t)+\mat{R}_0\cdot\vec{u}_{loc,1}(\vec{y},t)&=\vec{v}_1+B(\vec{\omega}_0)\cdot\mat{R}_0\cdot\vec{y},\\
  \diffD_2\vec{u}_0(\vec{c}_0(t),t)+\mat{R}_1\cdot\vec{u}_{loc,1}(\vec{y},t)+\mat{R}_0\cdot\vec{u}_{loc,2}(\vec{y},t)&=\vec{v}_2+(B(\vec{\omega}_0)\cdot\mat{R}_1+B(\vec{\omega}_1)\cdot\mat{R}_0)\cdot\vec{y}.
 \end{align*}

Inserting the asymptotic expansion in the momentum equation of the fluid and applying the chain rule for the local fields yields
\begin{align*}
\Re(\partial_t\vec{u}_0+(\vec{u}_0\cdot\nabla_{\vec{x}})\vec{u}_0)-\nabla_{\vec{x}}\cdot\mat{S}[\vec{u}_0]^T-\Re\Fr^{-2}\vec{e}_g=\mat{R}^{e}\cdot\nabla_{\vec{y}}\cdot\left(\epsilon^{-1}\mat{S}[\vec{u}_{loc,1}]^T+\mat{S}[\vec{u}_{loc,2}]^T\right)+\gO(\epsilon).
\end{align*}
Since $\vec{u}_0$ satisfies the Navier-Stokes equations and the local fields satisfy the stationary Stokes equations according to (R1) and (R3), the momentum balance holds up to an error of $\gO(\epsilon)$.

As for the momentum balances of the particle, we insert the asymptotic expansion on both sides of the equations. On the right-hand side we use a Taylor series for the term $\mat{S}[\vec{u}_0]$ in $\vec{c}_0$ and keep the integrals for the local fields, this leads to
\begin{subequations}\label{eq_asymAn}
\begin{align}
 &m\Re\epsilon^{k-2}\mat{R}_0^T\cdot\left(\frac{\diffd}{\diffd t}\vec{v}_0-\frac{1}{\Fr^2}\vec{e}_g+\epsilon\frac{\diffd}{\diffd t}\vec{v}_1\right)\nonumber\\
 &\qquad=\int_{\partial\aset{E}}\mat{S}[\vec{u}_{loc,1}]\cdot\vec{n}\,\diffd s+\epsilon\int_{\partial\aset{E}}\left(\mat{R}_0^T\cdot \diffD_1\mat{S}[\vec{u}_0]\cdot\mat{R}_0+\mat{R}_0^T\cdot\mat{R}_1\cdot\mat{S}[\vec{u}_{loc,1}]+\mat{S}[\vec{u}_{loc,2}]\right)\cdot\vec{n}\,\diffd s\label{eq_asymAn_intCond1}
\end{align}
for the linear momentum and 
\begin{align}
  &m\Re\epsilon^{k-1}\mat{R}_0^T\cdot\left(\frac{\diffd}{\diffd t}\left(\mat{R}_0\cdot\hat{\mat{J}}\cdot\mat{R}_0^T\cdot\vec{\omega}_0\right)+\epsilon \frac{\diffd}{\diffd t}\left(\mat{R}_0\cdot\hat{\mat{J}}\cdot\mat{R}_0^T\cdot\vec{\omega}_1+\left(\mat{R}_0\cdot\hat{\mat{J}}\cdot\mat{R}_1^T+\mat{R}_1\cdot\hat{\mat{J}}\cdot\mat{R}_0^T\right)\cdot\vec{\omega}_0\right)\right)\nonumber\\
  &=~\int_{\partial\aset{E}}\vec{y}\times\mat{S}[\vec{u}_{loc,1}]\cdot\vec{n}\,\diffd s\nonumber\\
  &\quad+\epsilon\left(\int_{\partial\aset{E}}\vec{y}\times\left(\mat{R}_0^T\cdot \diffD_1\mat{S}[\vec{u}_0]\cdot\mat{R}_0+\mat{S}[\vec{u}_{loc,2}]\right)\cdot\vec{n}\,\diffd s+\mat{R}_0^T\cdot\mat{R}_1\cdot\int_{\partial\aset{E}}\vec{y}\times\left(\mat{S}[\vec{u}_{loc,1}]\cdot\vec{n}\right)\,\diffd s\right)\label{eq_asymAn_intCond2}
\end{align}
\end{subequations}
for the angular momentum. These are exactly the integral conditions~\eqref{eq_asymAn_force} and~\eqref{eq_asymAn_couple}.

The requirement (R4) is needed to avoid contradictory conditions as we explain in Remark~\ref{rem_asymAn_impactOfScaling}.
\end{proof}

In Lemma~\ref{lem_asymAna_junk} we assume that (R3) is fulfilled despite the fact that the problem seems to be overdetermined at first glance by the presence of Dirichlet as well as integral conditions on $\partial\aset{E}$. This is however not true as the following lemma shows.

\begin{lem}[Solvability conditions of Stokes problem]\label{lem_asymAna_solvCond}
Let $\vec{w}_q$, $p_q$, $q=1,\ldots,6$, be the solutions of the following six stationary Stokes problems related to the boundary conditions of the three elementary translations and rotations:
\begin{align}\label{eq_asymAn_momentProblem}\nonumber
 \nabla_{\vec{y}}\cdot\mat{S}[\vec{w}_q]^T&=\vec{0},\qquad \nabla_{\vec{y}}\cdot\vec{w}_q=0,\hspace*{6ex} q=1,\ldots,6,\quad\vec{y}\in\R^3\setminus\overline{\aset{E}},\\\nonumber
 \vec{w}_q&=\vec{e}_q,\qquad \vec{w}_{q+3}=\vec{y}\times\vec{e}_q,\quad q=1,2,3,\hspace*{4.2ex}\vec{y}\in\partial\aset{E},\\
 \|\vec{w}_{q}\|&\leq c_q/\|\vec{y}\|,\qquad \|\nabla_{\vec{y}}\vec{w}_{q}\|,\|p_q\|\leq d_q/\|\vec{y}\|^2,\qquad\|\vec{y}\|\geq f_q>0
\end{align}
for constants $c_q$, $d_q$, $f_q$. Set the following surface moments associated with $\vec{w}_q$, $q=1,\ldots,6$ to be 
\begin{align*}
 \vec{s}_q&=\int_{\partial\aset{E}}\vec{y}\times\left(\mat{S}[\vec{w}_q]\cdot\vec{n}\right)\,\diffd s(\vec{y}),\qquad \vec{t}_q=\int_{\partial\aset{E}}\left(\mat{S}[\vec{w}_q]\cdot\vec{n}\right)\,\diffd s(\vec{y}),\\
  \mat{V}_q&=\int_{\partial\aset{E}}\vec{y}\otimes\left(\mat{S}[\vec{w}_q]\cdot\vec{n}\right)\,\diffd s(\vec{y}),\qquad W_q=\int_{\partial\aset{E}}\vec{y}\otimes\vec{y}\otimes\left(\mat{S}[\vec{w}_q]\cdot\vec{n}\right)\,\diffd s(\vec{y}).
\end{align*}
\begin{itemize}
 \item[1)] Let (R3) of Lemma~\ref{lem_asymAna_junk} be fulfilled, then the following solvability conditions hold:
\begin{subequations}\label{eq_asymAn_solvCondAll}
 \begin{align}
 \begin{pmatrix}\vec{s}_q\\\vec{t}_q\end{pmatrix}\cdot\begin{pmatrix}\mat{R}_0^T\cdot\vec{\omega}_{0}\\\mat{R}_0^T\cdot\vec{v}_1\end{pmatrix}&=\left(\mat{R}_0^T\cdot\partial_{\vec{x}}\vec{u}_0^T\cdot\mat{R}_0\right):\mat{V}_q+\left(\mat{R}_0^T\cdot\partial_{\vec{x}}\vec{u}_0\cdot\vec{c}_1\right)\cdot\vec{t}_q+\begin{cases}\vec{e}_q\cdot\vec{f}_1&q=1,2,3\\-\vec{e}_{q-3}\cdot\vec{g}_1&q=4,5,6\end{cases},\label{eq_asymAn_solvCond1}\\
 \begin{pmatrix}\vec{s}_q\\\vec{t}_q\end{pmatrix}\cdot\begin{pmatrix}\mat{R}_0^T\cdot\vec{\omega}_{1}\\\mat{R}_0^T\cdot\vec{v}_2\end{pmatrix}&=\left(\mat{R}_1^T\cdot\left( B(\vec{\omega}_0)+\partial_{\vec{x}}\vec{u}_0^T\right)\cdot\mat{R}_0-\mat{R}_0^T\cdot\left(B(\vec{\omega}_0)+\partial_{\vec{x}}\vec{u}_0^T\right)\cdot\mat{R}_0\cdot\mat{R}_1^T\cdot\mat{R}_0+\mat{L}\right):\mat{V}_q\nonumber\\
 &\quad+\left(\mat{R}_0^T\cdot\left(\mat{R}_1\cdot\mat{R}_0^T\cdot\left(\vec{v}_1-\partial_{\vec{x}}\vec{u}_0\cdot\vec{c}_1\right)+0.5[\vec{c}_1\otimes\vec{c}_1:\nabla^2]\vec{u}_0+\partial_{\vec{x}}\vec{u}_0\cdot\vec{c}_2\right)\right)\cdot\vec{t}_q\nonumber\\
 &\quad+\sum_{k,\ell,m}K_{k\ell m}(W_q)_{k\ell m}+\begin{cases}\vec{e}_q\cdot\vec{f}_2&q=1,2,3\\-\vec{e}_{q-3}\cdot\vec{g}_2&q=4,5,6\end{cases}.\label{eq_asymAn_solvCond2}
\end{align}
\end{subequations}
with $(\mat{R}_0)_{ij}=\vec{e}_i\cdot\vec{p}_j$, $(\mat{L})_{ij}=[\vec{p}_{i}\otimes\vec{c}_1:\nabla^2](\vec{u}_0\cdot\vec{p}_{j})$ and $K_{k\ell m}=0.5[\vec{p}_k\otimes\vec{p}_\ell:\nabla^2](\vec{u}_0\cdot\vec{p}_m)$.
 \item[2)] Let \eqref{eq_asymAn_solvCondAll} hold. Let $(\vec{u}_{loc,i}, p_{loc,i})$ be the solution of the stationary Stokes problem \eqref{eq_asymAn_statStokesBasic} with the Dirichlet condition \eqref{eq_asymAn_dirichlet} and with the decay property \eqref{eq_asymAn_decay}, then the integral conditions \eqref{eq_asymAn_force} and \eqref{eq_asymAn_couple} are fulfilled.
 \item[3)] Let \eqref{eq_asymAn_solvCondAll} hold. Let $(\vec{u}_{loc,i}, p_{loc,i})$ be the solution of the stationary Stokes problem \eqref{eq_asymAn_statStokesBasic} with the integral conditions \eqref{eq_asymAn_force} and \eqref{eq_asymAn_couple} and with the decay property \eqref{eq_asymAn_decay}, then the Dirichlet condition \eqref{eq_asymAn_dirichlet} is fulfilled.
 \item[4)] The linear systems \eqref{eq_asymAn_solvCondAll} are invertible.
\end{itemize}
\end{lem}

\begin{proof}
The key ingredient for this proof is the existence of the Green formula for solutions of the Stokes problem (\cite{Junk2007}, Lemma~6), which connects the stresses and boundary conditions of two Stokes solutions and which needs the decay properties~\eqref{eq_asymAn_decay} as well as those in~\eqref{eq_asymAn_momentProblem}.

For 1): The equations~\eqref{eq_asymAn_force} and~\eqref{eq_asymAn_couple} are equivalent to
\begin{subequations}\label{eq_asymAn_gettingSC1}
\begin{align}
 \vec{e}_q\cdot\vec{f}_i&=\vec{e}_q\cdot\int_{\partial\aset{E}}\mat{S}[\vec{u}_{loc,i}]\cdot\vec{n}\,\diffd s=\int_{\partial\aset{E}}\vec{w}_q\cdot\left(\mat{S}[\vec{u}_{loc,i}]\cdot\vec{n}\right)\,\diffd s,\\\nonumber
 \vec{e}_q\cdot\vec{g}_i&=\vec{e}_q\cdot\int_{\partial\aset{E}}B(\vec{y})\cdot\mat{S}[\vec{u}_{loc,i}]\cdot\vec{n}\,\diffd s=\int_{\partial\aset{E}}\left(B(\vec{y})^T\cdot\vec{e}_q\right)\cdot\left(\mat{S}[\vec{u}_{loc,i}]\cdot\vec{n}\right)\,\diffd s\\
 &=-\int_{\partial\aset{E}}\left(\vec{y}\times\vec{e}_q\right)\cdot\left(\mat{S}[\vec{u}_{loc,i}]\cdot\vec{n}\right)\,\diffd s=-\int_{\partial\aset{E}}\vec{w}_{q+3}\cdot\left(\mat{S}[\vec{u}_{loc,i}]\cdot\vec{n}\right)\,\diffd s
\end{align}
\end{subequations}
for $q=1,2,3$. With the help of the Green formula the roles of $\vec{w}_q$ and $\vec{u}_{loc,i}$ in the right-hand sides of~\eqref{eq_asymAn_gettingSC1} can be interchanged. Thus \eqref{eq_asymAn_gettingSC1} is equivalent to
\begin{align}\label{eq_asymAn_gettingSC2}
 \vec{e}_q\cdot\vec{f}_i=\int_{\partial\aset{E}}\vec{h}_i\cdot\mat{S}[\vec{w}_q]\cdot\vec{n}\,\diffd s,\qquad-\vec{e}_q\cdot\vec{g}_i=\int_{\partial\aset{E}}\vec{h}_i\cdot\mat{S}[\vec{w}_{q+3}]\cdot\vec{n}\,\diffd s,\quad q=1,2,3.
\end{align}
We note that for $i=1,2$ the terms $\vec{h}_i$ of Abbreviation~\ref{abb_asymAn_1} have the form 
\begin{align*}
 \vec{h}_i=\mat{R}_0^T\cdot\vec{v}_i+\mat{R}_0^T\cdot B(\vec{\omega}_{i-1})\cdot\mat{R}_0\cdot\vec{y}-\vec{r}_i=\mat{R}_0^T\cdot\vec{v}_i+B(\mat{R}_0^T\cdot\vec{\omega}_{i-1})\cdot\vec{y}-\vec{r}_i,
\end{align*}
where we used $\mat{M}^T\cdot B(\vec{\omega})\cdot\mat{M}=B(\mat{M}^T\cdot\vec{\omega})$, which holds for any $\mat{M}\in\SO(3)$ and any $\vec{\omega}\in\R^3$, and where $\vec{r}_1=\mat{R}_0^T\cdot \diffD_1\vec{u}_0$ and $\vec{r}_2=\mat{R}_0^T\cdot\left(\mat{R}_1\cdot\vec{u}_{loc,1}+\diffD_2\vec{u}_0-B(\vec{\omega}_0)\cdot\mat{R}_1\cdot\vec{y}\right)$. Inserting $\vec{h}_i$ and sorting the resultant terms with respect to $\vec{\omega}_{i-1}$ and $\vec{v}_i$ leads to
\begin{align}
 \left(\mat{R}_0^T\cdot\vec{\omega}_{i-1}\right)\cdot\vec{s}_q+\left(\mat{R}_0^T\cdot\vec{v}_i\right)\cdot\vec{t}_q&=\int_{\partial\aset{E}} \vec{r}_i\cdot\mat{S}[\vec{w}_q]\cdot\vec{n}\,\diffd s(\vec{y})+\begin{cases}\vec{e}_q\cdot\vec{f}_i&q=1,2,3\\-\vec{e}_{q-3}\cdot\vec{g}_i&q=4,5,6\end{cases},\label{eq_asymAn_solvCond}
\end{align}
which directly results in~\eqref{eq_asymAn_solvCondAll}.

For 2) we consider~\eqref{eq_asymAn_solvCondAll} which is equivalent to~\eqref{eq_asymAn_solvCond} and thus to~\eqref{eq_asymAn_gettingSC2}, since $\vec{h}_i$ is the Dirichlet condition of $\vec{u}_{loc,i}$ by assumption. We use the Green formula and get~\eqref{eq_asymAn_gettingSC1}.

For 3) we analogously arrive at~\eqref{eq_asymAn_gettingSC2}. The assumption that the integral conditions are fulfilled is equivalent to~\eqref{eq_asymAn_gettingSC1}. Combining both results gives the equalities
\begin{align*}
 \int_{\partial\aset{E}}\vec{w}_q\cdot\left(\mat{S}[\vec{u}_{loc,i}]\cdot\vec{n}\right)\,\diffd s&=\int_{\partial\aset{E}}\vec{h}_i\cdot\mat{S}[\vec{w}_q]\cdot\vec{n}\,\diffd s,\\
 \int_{\partial\aset{E}}\vec{w}_{q+3}\cdot\left(\mat{S}[\vec{u}_{loc,i}]\cdot\vec{n}\right)\,\diffd s&=\int_{\partial\aset{E}}\vec{h}_i\cdot\mat{S}[\vec{w}_{q+3}]\cdot\vec{n}\,\diffd s.
\end{align*}
Since $\vec{u}_{loc,i}$ and $\vec{w}_q$ fulfill the decay property the above equations imply $\vec{u}_{loc,i}=\vec{h}_i$ on $\partial\aset{E}$ by means of the Green formula.

For 4): In (\cite{Junk2007}, Lemma~8) it is shown that the six vectors $(\vec{t}_q,\vec{s}_q)$, $q=1,\ldots,6$ are linearly independent, thus the system \eqref{eq_asymAn_solvCondAll} is invertible for all right-hand sides.
\end{proof}

\begin{remark}\hfill
\begin{itemize}
 \item[1)] 
 The solvability conditions~\eqref{eq_asymAn_solvCondAll} together with the particle related equations~\eqref{eq_asymAn_particleEqs} build a system of nonlinear ordinary differential equations (ODEs) for $\vec{c}_i$ and $\mat{R}_{i-1}$, $i=1,2$, where the ODE for the $i$th asymptotic coefficient depends on lower order coefficients, while the surface moments depend only on the geometry of the particle. The order of the corresponding ODEs depends on the choice of the mass function $\alpha_{mass}$. For example in the case $\alpha_{mass}(\epsilon)=\epsilon$, $\vec{f}_1$ depends on $\diffd\vec{v}_1/\diffd t$ and thus \eqref{eq_asymAn_solvCond1} becomes second order for $\vec{c}_1$ with an additional initial condition $\vec{v}_1(0)=\vec{0}$. However, under suitable regularity assumptions on $\vec{u}_0$, the corresponding right-hand sides of the ODE systems are at least continuous and thus at least local solutions exist by Theorem of Peano.
 \item[2)] 
 Solving the solvability conditions of the Stokes problems~\eqref{eq_asymAn_solvCondAll} determine the Dirichlet conditions~\eqref{eq_asymAn_dirichlet}. However, it is not our aim to examine the existence, uniqueness and regularity of the Stokes solutions given by~\eqref{eq_asymAn_dirichlet},~\eqref{eq_asymAn_decay} in general in this paper. Instead we will discuss some special problems in Section~\ref{sec_4_specCase} where analytical solutions are available.
 \item[3)] 
 Just like in \cite{Junk2007}, Lemma~\ref{lem_asymAna_junk} and~\ref{lem_asymAna_solvCond} imply a procedure for constructing a $\gO(\epsilon)$-solution of~\eqref{eq_mathMod_dimFreeSystem}: First the Navier-Stokes equations \eqref{eq_asymAn_navStokesMomentum} are solved to get $\vec{u}_0$ and $p_0$, then the path $\vec{c}_0$ of the particle is obtained from the tracer condition~\eqref{eq_asymAn_tracerCond}. The solution of the solvability conditions~\eqref{eq_asymAn_solvCond1} together with the corresponding differential equations~\eqref{eq_asymAn_particleEqs} for $i=1$ provides the coefficients $\vec{c}_1$, $\vec{v}_1$ and $\vec{\omega}_0$, $\mat{R}_0$ that define the Dirichlet condition~\eqref{eq_asymAn_dirichlet}. The stationary Stokes problem~\eqref{eq_asymAn_statStokesComplete} for $i=1$ is solved. Finally, the last two steps are repeated to get $\vec{u}_{loc,2}$ which completes the procedure.
\end{itemize}
\end{remark}

\begin{remark}[Density ratio scaling]\label{rem_asymAn_impactOfScaling}
We want to point out the effect of the density ratio scaling introduced in~\eqref{eq_mathMod_massfcn}. The momentum balances of the particle \eqref{eq_asymAn} can be written as 
 \begin{align*}
  \epsilon^{k-2}(\vec{k}_0+\epsilon\vec{k}_1)&=\vec{m}_1^v+\epsilon\vec{m}_2^v,\\
  \epsilon^{k-1}(\vec{\ell}_0+\epsilon\vec{\ell}_1)&=\vec{m}_1^{\omega}+\epsilon\vec{m}_2^{\omega},
 \end{align*}
with $\vec{k}_i$, $\vec{\ell}_i$ of Abbreviation~\ref{abb_asymAn_1} and $\vec{m}_{i+1}^v,\vec{m}_{i+1}^{\omega}$ denoting the according integral terms appearing on the right-hand sides of~\eqref{eq_asymAn_intCond1} and~\eqref{eq_asymAn_intCond2}. The terms $\vec{k}_i,\vec{\ell}_i$ describe the linear and angular accelerations of the particle, while the integral terms connect them to the Dirichlet conditions of the corresponding local flow fields via the solvability conditions~\eqref{eq_asymAn_solvCondAll}. Different choices of the mass function lead to the following pattern for the balancing of the accelerations with the integral terms:
\begin{align}\label{eq_asymAn_tableOfRegimesLin}
 \setlength{\arraycolsep}{2.5ex}
 \begin{array}{c|ccccc}
  \ell\backslash k&0&1&2&3&4\\\hline
  -2&\vec{k}_0=\vec{0}&\vec{0}=\vec{0}&\vec{0}=\vec{0}&\vec{0}=\vec{0}&\vec{0}=\vec{0}\\
  -1&\vec{k}_1=\vec{0}&\vec{k}_0=\vec{0}&\vec{0}=\vec{0}&\vec{0}=\vec{0}&\vec{0}=\vec{0}\\
  0&...=\vec{m}_1^v&\vec{k}_1=\vec{m}_1^v&\vec{k}_0=\vec{m}_1^v&\vec{0}=\vec{m}_1^v&\vec{0}=\vec{m}_1^v\\
  1&...=\vec{m}_2^v&...=\vec{m}_2^v&\vec{k}_1=\vec{m}_2^v&\vec{k}_0=\vec{m}_2^v&\vec{0}=\vec{m}_2^v\\
  2&...=...&...=...&...=...&\vec{k}_1=...&\vec{k}_0=...\\
  3&...=...&...=...&...=...&...=...&\vec{k}_1=...
 \end{array}
\end{align}
 \begin{align}\label{eq_asymAn_tableOfRegimesAng}
 \setlength{\arraycolsep}{2.5ex}
 \begin{array}{c|ccccc}
  \ell\backslash k&0&1&2&3&4\\\hline
  -1&\vec{\ell}_0=\vec{0}&\vec{0}=\vec{0}&\vec{0}=\vec{0}&\vec{0}=\vec{0}&\vec{0}=\vec{0}\\
  0&\vec{\ell}_1=\vec{m}_1^{\omega}&\vec{\ell}_0=\vec{m}_1^{\omega}&\vec{0}=\vec{m}_1^{\omega}&\vec{0}=\vec{m}_1^{\omega}&\vec{0}=\vec{m}_1^{\omega}\\
  1&...=\vec{m}_2^{\omega}&\vec{\ell}_1=\vec{m}_2^{\omega}&\vec{\ell}_0=\vec{m}_2^{\omega}&\vec{0}=\vec{m}_2^{\omega}&\vec{0}=\vec{m}_2^{\omega}\\
  2&...=...&...=...&\vec{\ell}_1=...&\vec{\ell}_0=...&\vec{0}=...\\
  3&...=...&...=...&...=...&\vec{\ell}_1=...&\vec{\ell}_0=...\\
  4&...=...&...=...&...=...&...=...&\vec{\ell}_1=...
 \end{array}
\end{align}
Here each row shows the $\gO(\epsilon^\ell)$-correction of the momentum equations and each column corresponds to the choice of power of $\epsilon$ of the mass function.

For $k\leq1$ the condition $\diffd\vec{v}_0/\diffd t=\Fr^{-2}\vec{e}_g$ arises (see~\eqref{eq_asymAn_tableOfRegimesLin}). In combination with the tracer condition~\eqref{eq_asymAn_tracerCond}, which is independent of the mass function $\alpha_{mass}$, this leads to the implication that $\diffd\vec{u}_0(\vec{c}_0(t),t)/\diffd t=\Fr^{-2}\vec{e}_g$ in contradiction to the claim, $\vec{u}_0$ solving the Navier-Stokes equations in $\Omega$, according to (R1) of Lemma~\ref{lem_asymAna_junk}. This fact indicates that in this regime Assumption~\ref{assum_mathMod_assumptions1}, 1) does not hold anymore, since the inertial effects of the particle are too strong and the perturbation of the surrounding fluid is of the same order as $\vec{u}_0$. Thus, in consistence to the asymptotic, we restrict our classification of inertial types in Table~\ref{tab_mathMod_inertReg} to $k\geq 2$.
\end{remark}

\subsection{Extension of Stokes solutions to the particle domain}\label{subsec_defi_stokes_interior}

The formulation of the suspension model requires the definition of velocities and stresses on the whole domain $\Omega$, including the interior of the particle. Hence, we need to introduce a meaningful extension of the quantities to the particle domain $\aset{E}$. A possibility is to consider $\aset{E}$ as a part of the fluid domain and use the so-called singularity solutions which express the Stokes solutions by means of the Green's dyadic (Oseen-Burgers tensor) \cite{Kim2005}. Such solutions have the disadvantage of not being bounded in the neighborhood of the origin, which is an important requirement in our modeling as we will see in Section~\ref{sec_3_kinMod}. Alternatively, one could also think of a continuous extension of the Dirichlet conditions \eqref{eq_asymAn_dirichlet} to $\aset{E}$ and model the stresses as Newtonian. This leads to a discontinuity in the stresses at the particle boundary which contradicts with the desire for a smooth macroscopic stress on $\Omega$. Therefore, we follow here the approach of \cite{Patankar2000} and treat the particle as a fluid with a rigidity constraint. Proceeding from a dimensional description (analogously to Section~\ref{subsec_mathMod_threedimInterface}), we obtain a consistent model for the dimensionless stresses and velocities in the particle domain by help of the presented asymptotic techniques.

The fluid in the particle domain is characterized by the velocity $\vec{z}:\aset{E}(t)\times\R_0^+\to\R^3$  of a material point and the non-Newtonian stress $\mat{T}:\aset{E}(t)\times\R_0^+\to\R^{3\times3}$, i.e.\ 
\begin{subequations}\label{eq_defiStokesParticle_full}
\begin{align}\label{eq_defiStokesParticle}
 \rho_p\left(\partial_t\vec{z}+(\vec{z}\cdot\nabla)\vec{z}\right)&=\nabla\cdot\mat{T}^T+\rho_pg\vec{e}_g,\qquad\nabla\vec{z}+\nabla\vec{z}^T=\vec{0}, && \vec{x}\in\aset{E}(t),\\
 \vec{z}&=\vec{u},\qquad \mat{T}\cdot\vec{n}=\mat{S}[\vec{u}]\cdot\vec{n}, && \vec{x}\in\partial\aset{E}(t),
\end{align}
\end{subequations}
for all $t>0$ with the flow velocity $\vec{u}$ solving \eqref{eq_mathMod_NSE_dimensional} and appropriate initial conditions in consistency with \eqref{eq_mathMod_full_dimensional}. The stress $\mat{T}$ acts as Lagrange multiplier to the rigidity constraint in \eqref{eq_defiStokesParticle} and can be expressed as symmetric gradient field of the three-dimensional unknown $\vec{\lambda}:\aset{E}(t)\times\R_0^+\to\R^3$, i.e.\ $\mat{T}=\mat{T}[\vec{\lambda}]=\nabla\vec{\lambda}+\nabla\vec{\lambda}^T$ \cite{Patankar2000}.\footnote{The following considerations can equivalently be formulated for $\vec{\lambda}$, but since we are only interested in the stresses themselves, we do not use the inner structure of $\mat{T}$ explicitly. It is however important to note that $\mat{T}$ have only three degrees of freedom such that the subsequent boundary value problems are well-posed.} The velocity in the rigid body domain can be explicitly stated in terms of the particle quantities (center of mass $\vec{c}$, linear and angular velocities $\vec{v}$, $\vec{\omega}$ of \eqref{eq_mathMod_full_dimensional}) as $\vec{z}=B(\vec{\omega})\cdot(\vec{x}-\vec{c})+\vec{v}$. Using this expression and the ODE for the director triad~\eqref{eq_mathMod_R_dimensional} simplifies~\eqref{eq_defiStokesParticle_full} to a boundary value problem for $\mat{T}$, resp. $\vec{\lambda}$:
\begin{align*}
 \rho_p\left(\left(\frac{\diffd^2}{\diffd t^2}\mat{R}\right)\cdot\mat{R}^T\cdot(\vec{x}-\vec{c})+\frac{\diffd}{\diffd t}\vec{v}\right)&=\nabla\cdot\mat{T}[\vec{\lambda}]^T+\rho_pg\vec{e}_g, && \vec{x}\in\aset{E}(t),\\
 \mat{T}[\vec{\lambda}]\cdot\vec{n}&=\mat{S}[\vec{u}]\cdot\vec{n},&&\vec{x}\in\partial\aset{E}(t).
\end{align*}
To transform the problem to the particle reference state we set $\mat{T}(\vec{x},t)=\mat{R}(t)\cdot\tilde{\mat{T}}(\vec{y}(\vec{x},t),t)\cdot\mat{R}^T(t)$. In dimensionless form it is given by 
\begin{subequations}\label{eq_defiStokesParticle_dimless}
\begin{align}
 \rho\Re\left(\frac{\diffd^2}{\diffd t^2}\left(\epsilon\mat{R}\cdot\vec{y}+\vec{c}\right)-\Fr^{-2}\vec{e}_g\right)&=\epsilon^{-1}\mat{R}\cdot\nabla_{\vec{y}}\cdot\tilde{\mat{T}}^T,&& \vec{y}\in\aset{E},\label{eq_defiStokesParticle_dimless_1}\\
 \tilde{\mat{T}}\cdot\vec{n}&=\mat{R}^T\cdot\mat{S}[\vec{u}](\vec{x}(\vec{y},t,\epsilon),t)\cdot\mat{R}\cdot\vec{n},&& \vec{y}\in\partial\aset{E}.\label{eq_defiStokesParticle_dimless_2}
\end{align}
\end{subequations}
We model $\tilde{\mat{T}}$ as composition of the Newtonian stress of $\vec{u}_0$ and some disturbance, $\tilde{\mat{T}}(\vec{y},t,\epsilon)=\mat{R}^{T}(t)\cdot\mat{S}[\vec{u}_0](\vec{x}(\vec{y},t,\epsilon),t)\cdot\mat{R}(t)+\mat{T}_{loc}(\vec{y},t,\epsilon)$. As asymptotic expansion we take a regular power series in the size parameter $\epsilon$ for $\mat{T}_{loc}$ and use a Taylor series in $\vec{c}_0$ for $\mat{S}[\vec{u}_0](\vec{x}(\vec{y},t,\epsilon),t)$. With
\begin{align*}
 \tilde{\mat{T}}^e = \mat{T}_{loc,1}+\mat{R}_0^T\cdot\mat{S}[\vec{u}_0]\cdot\mat{R}_0+\epsilon\left(\mat{T}_{loc,2}+\mat{R}_0^T\cdot D_1\mat{S}[\vec{u}_0]\cdot\mat{R}_0+\mat{R}_1^T\cdot\mat{S}[\vec{u}_0]\cdot\mat{R}_0+\mat{R}_0^T\cdot\mat{S}[\vec{u}_0]\cdot\mat{R}_1\right),
\end{align*}
we can formulate the asymptotic model for the stresses in Lemma~\ref{lem_defiStokesParticle_asymptoticModel}.

\begin{lem}[Asymptotic model of stresses in particle domain]\label{lem_defiStokesParticle_asymptoticModel}
Let the requirements of Lemma~\ref{lem_asymAna_junk} be fulfilled. Let $\mat{T}_{loc,i}=\nabla_{\vec{y}}\vec{\lambda}_{loc,i}+\nabla_{\vec{y}}\vec{\lambda}_{loc,i}^T$, $i=1,2$ be the solutions of the following boundary value problems for $t>0$
 \begin{subequations}
 \begin{align}
  \nabla_{\vec{y}}\cdot\mat{T}_{loc,i}^T &= \hat{\vec{f}}_i, && \vec{y}\in\aset{E},\label{eq_defiStokesParticle_forceDef}\\
  \mat{T}_{loc,i}\cdot\vec{n} &= \mat{S}[\vec{u}_{loc,i}]\cdot\vec{n}, && \vec{y}\in\partial\aset{E},\label{eq_defiStokesParticle_dirichletStress}
 \end{align}
 \end{subequations}
 where the force terms are given by
 \begin{align*}
  \hat{\vec{f}}_1 &= \begin{cases}\vec{0},&k\geq3\\|\aset{E}|^{-1}\vec{k}_0,&k=2\end{cases},\\
  \hat{\vec{f}}_2 &=-\mat{R}_0^T\cdot\nabla_{\vec{x}}\cdot\mat{S}[\vec{u}_0]^T-\mat{R}_0^T\cdot\mat{R}_1\cdot\hat{\vec{f}}_1+\begin{cases}\vec{0},&k\geq4\\|\aset{E}|^{-1}\vec{k}_0,&k=3\\|\aset{E}|^{-1}m\Re\mat{R}_0^T\cdot\left(\frac{\diffd}{\diffd t}\vec{v}_1+\frac{\diffd^2}{\diffd t^2}\mat{R}_0\cdot\vec{y}\right),&k=2\end{cases}.
 \end{align*}
Then $\tilde{\mat{T}}^e$, $\mat{R}^e$, $\vec{c}^e$ solve the system~\eqref{eq_defiStokesParticle_dimless} at least up to an error of $\gO(\epsilon)$.
\end{lem}
\begin{proof}
Similarly to the proof of Lemma~\ref{lem_asymAna_junk}, the approximation order follows from a straight forward computation.

Inserting $\tilde{\mat{T}}^e$ on the left side of~\eqref{eq_defiStokesParticle_dimless_2}, using the expansion described in the proof of Lemma~\ref{lem_asymAna_junk} for $\mat{S}[\vec{u}]$ on the right-hand side and taking~\eqref{eq_defiStokesParticle_dirichletStress} into account yields that~\eqref{eq_defiStokesParticle_dimless_2} holds up to $\gO(\epsilon^2)$.

When taking the divergence of $\tilde{\mat{T}}^e$ with respect to $\vec{y}$, the terms of the form $\mat{R}_i\cdot\mat{S}[\vec{u}_0]\cdot\mat{R}_j$ vanish, since the Newtonian stresses are evaluated at $\vec{c}_0$. With the asymptotic expansions of $\mat{R}$ and $\vec{c}$, \eqref{eq_defiStokesParticle_dimless_1} is fulfilled up to $\gO(\epsilon)$ provided that \eqref{eq_defiStokesParticle_forceDef} holds.
\end{proof}

\begin{remark}[Local velocity fields and forces in particle domain]\label{rem_defiStokesParticle_1}\hfill~
\begin{itemize}
 \item[1)]
By construction, the stress $\mat{T}$ induces a rigid body velocity $\vec{z}=\epsilon(\diffd\mat{R}/\diffd t)\cdot\vec{y}+\vec{v}$ in $\aset{E}$, while its approximation $\mat{T}^e = \mat{R}^e\cdot\tilde{\mat{T}}^e\cdot\mat{R}^{eT}$ results in the approximation of $\vec{z}$, namely
\begin{align*}
\vec{z}^e=\vec{v}_0+\epsilon\left(\frac{\diffd}{\diffd t}\mat{R}_0\cdot\vec{y}+\vec{v}_1\right)+\epsilon^2\left(\frac{\diffd}{\diffd t}\mat{R}_1\cdot\vec{y}+\vec{v}_2\right),\qquad\vec{y}\in\aset{E}.
\end{align*}
Since the functions $\vec{v}_i$ and $\mat{R}_i$ do not depend on $\vec{y}$, we can rewrite $\vec{z}^e$ as
\begin{align*}
\vec{z}^e=\vec{u}_0(\vec{x}(\vec{y},t,\epsilon),t)+\epsilon\mat{R}^e(\vec{h}_1(\vec{y},t)+\epsilon\vec{h}_2(\vec{y},t))+\gO(\epsilon^3),
\end{align*}
using Abbreviation~\ref{abb_asymAn_1}, where $\vec{h}_i(\vec{y},t)$ is the continuous extension of the Dirichlet conditions of $\vec{u}_{loc,i}$ from Lemma~\ref{lem_asymAna_junk} to $\aset{E}$. This way we can extend the local disturbance velocity fields to $\aset{E}$ by setting $\vec{u}_{loc,i}=\vec{h}_{loc,i}$. The so defined velocity is bounded by the assumptions of Lemma~\ref{lem_asymAna_junk}.
\item[2)]
The force terms $\hat{\vec{f}}_i$ are not equal to $\vec{f}_i$ of Abbreviation~\ref{abb_asymAn_1}, but $\int_{\aset{E}}\hat{\vec{f}}_i\diffd\vec{y}=\vec{f}_i$ holds true in consistency with \eqref{eq_defiStokesParticle_dirichletStress}.
\end{itemize}
\end{remark}

%%%%%%%%%%%%%%%%%%%%%%%%%%%% section 2 %%%%%%%%%%%%%%%%%%%%%%%%%%%%%%%%%%%%%%%%%%
%%%%%%%%%%%%%%%%%%%%%%%%%%%%%%%%%%%%%%%%%%%%%%%%%%%%%%%%%%%%%%%%%%%%%%%%%%%%%%%%

\setcounter{equation}{0} \setcounter{figure}{0}
\section{Kinetic model for a particle suspension}\label{sec_3_kinMod}

In this section we present our main result, the asymptotical (macroscale) description for a particle suspension under consideration of inertial effects. For this purpose we set up mass and momentum balances by considering the particle suspension as a stochastic homogeneous fluid whose behavior is characterized by non-Newtonian stresses. To obtain a model for the qualitative behavior of this fluid, which is independent of the specific realization, we analyze the averaged equations. With the help of the ergodicity assumption we are able to derive an analytical expression for the bulk stresses. This procedure goes originally back to Batchelor in~\cite{Batchelor1970}, the novelty of our work is the use of rigorous asymptotical results as model ingredients as well as the regard of inertial effects. These effects involve additional net volume forces and arise from the small deviations of the particles motions from the streamlines of the surrounding fluid.

\subsection{General framework}\label{subsec_kinMod_general}

Embedded into Batchelor's framework \cite{Batchelor1970} of suspension modeling, we introduce the necessary modifications to address the extra stresses induced by the particles' inertia. 

We describe the suspension as a homogeneous fluid whose microscale properties vary stochastically depending on the initial configuration of the positions and orientations of the immersed particles. Given the probability space $(\aset{W},\aset{A},\mu)$, let $\vec{w}(\ldotp,\ldotp,\omega):\Omega\times\R_0^+\to\R^3$ and $q(\ldotp,\ldotp,\omega):\Omega\times\R_0^+\to\R$ be velocity and pressure fields, $\omega \in \aset{W}$. The particle suspension is assumed to be described by $(\vec{w},q)$ that is almost surely a solution of the incompressible flow problem
\begin{align*}
  \Re\left(\partial_t\vec{w}+(\vec{w}\cdot\nabla_{\vec{x}})\vec{w}\right)&=\nabla_{\vec{x}}\cdot\mat{\Sigma}^T+\Re\Fr^{-2}\vec{e}_g,\qquad \nabla_{\vec{x}}\cdot\vec{w}=0, &&\text{for }(\vec{x},t)\in\Omega\times\R^+,
\end{align*}
supplemented with appropriate initial, boundary and $\mat{\Sigma}$-related closure conditions. As in statistic turbulence modeling (see e.g.~\cite{Wilcox1993}) we decompose the random velocity field into a mean and a fluctuating part with expectation $\E[.]$
\begin{align*}
\vec{w}(\vec{x},t,\omega)=\vec{u}(\vec{x},t)+\vec{u}'(\vec{x},t,\omega), \qquad \E[\vec{w}]=\vec{u},
\end{align*}
analogously for the pressure $q=p+p'$. The task is to model the bulk stress $\mat{\Sigma}$ or the inner force $\nabla_{\vec{x}}\cdot\mat{\Sigma}^T$, respectively. We particularly split the force into a divergence of surface stresses and a body force that is generated by the fluctuations of the translational velocity of the immersed particles
\begin{align}\label{eq_split}
\nabla_{\vec{x}}\cdot\mat{\Sigma}^T=\nabla_{\vec{x}}\cdot{\mat{\Sigma}^s}^T+\vec{b}.
\end{align}
The viscous stresses are covered here by $\mat{\Sigma}^s$. The resulting averaged description of the suspension in $\Omega\times \R^+$ is then given by 
\begin{align}\label{eq_kinMod_kinmodel}
 \Re\left(\partial_t\vec{u}+(\vec{u}\cdot\nabla_{\vec{x}})\vec{u}\right)&=\nabla_{\vec{x}}\cdot\E[\mat{\Sigma}^s]^T+\E[\vec{b}]-\Re\nabla_{\vec{x}}\cdot\E[\vec{u}'\otimes\vec{u}']^T+\Re\Fr^{-2}\vec{e}_g, \quad
 \nabla_{\vec{x}}\cdot\vec{u}=0
\end{align}
(an analogue in turbulence modeling are the Reynolds-Averaged-Navier-Stokes equations).

\begin{assumption}[Suspension properties]\label{assum_kinMod_assumptions1}\hfill
 \begin{itemize}
  \item[(A1)]  The suspension consists of similar, independent and identically distributed particles $\aset{E}_k(t)$.
  \item[(A2)]  The suspension is dilute, i.e.\ the mean particle spacing is sufficiently big.
  \item[(A3)]  Assumption~\ref{assum_mathMod_assumptions1} (on particle-flow interactions) applies to every particle.
  \item[(A4)]  No outer forces or moments act on the particles except of the gravitational force.
  \item[(A5)]  The characteristic length scale of the fluid is much bigger than the length scale of the particle and the mean particle spacing.
  \item[(A6)]  The fluctuations are generated by superposition of the local disturbance fields of the particles.
  \item[(A7)]  The suspension is Newtonian up to an error of $\gO(\epsilon)$.
  \item[(A8)]  The suspension is locally statistically homogeneous.
  \item[(A9)] The ergodicity hypothesis holds, i.e.\ for any involved random function $f:\Omega\times\R^+\times\aset{W}\to\R^n$, $n\geq1$, there is almost surely an equality of the expectation and the integral of this function over a suitable space region $\aset{V}(\vec{x},t)\subset\Omega$:
  \begin{align*}
   \E[f](\vec{x},t) = \int_{\aset{W}}f(\vec{x},t,\omega) \diffd\mu(\omega) \stackrel{!}{=} \frac{1}{\aset{V}(\vec{x},t)}\int_{\aset{V}(\vec{x},t)}f(\vec{\xi},t,\omega)\diffd\vec{\xi} = \E[f]^{\aset{V}}(\vec{x},t).
  \end{align*}
 \end{itemize}
\end{assumption}

Most of these assumptions were already used in \cite{Batchelor1970}: With (A1) and (A2) particle interactions that arise if two particles are close together can be neglected. Assumptions (A8) and (A9) allow the switching from the abstract mean value to the analytically powerful volume average. To use the results of the asymptotic analysis we impose (A3)-(A7). In contrast to \cite{Batchelor1970} we especially presuppose (A7) in order to consistently apply the asymptotical one-particle model of Section~\ref{sec_1_mathMod}. Otherwise we would need to consider a non-Newtonian stress term of the surrounding fluid in the asymptotics. Note that the introduced stress splitting \eqref{eq_split} becomes essential for the proper handling of inertial particles when dealing with volume averages. Whereas the expectation $\E[.]$ is linear, the permutability of  $\E[.]^{\aset{V}}$ and the divergence operator is generally not valid, see also Remark~\ref{rem_specCase_motSplitting} in Section~\ref{sec_4_specCase}.

\subsection{Macroscopic stresses and forces}

The core of the suspension model \eqref{eq_kinMod_kinmodel} are the macroscopic stresses and forces that we deduce from the asymptotic one-particle model. We realize the fluctuation related quantities that are marked with the index $'$ by superposing the local disturbance fields $._{loc}$ of the particles (Assumption~\ref{assum_kinMod_assumptions1}, (A6)).

\subsubsection*{Surface stresses}
The derivation of the model for $\mat{\Sigma}^s$ goes along \cite{Batchelor1970}. It is based on the idea of applying the ergodicity hypothesis for $\E[\mat{\Sigma}^s]$ and thus integrating the stresses over a suitable averaging volume $\aset{V}(\vec{x},t)$ which contains $N(\vec{x},t)$ particles $\aset{E}_k(t)$, $k=1,\ldots,N$. 
The underlying assumption on $\E[\mat{\Sigma}^s]^{\aset{V}}(\vec{x},t)$ is
\begin{align}\label{eq_Sigmas}
 \E[\mat{\Sigma}^s]^{\aset{V}} = \mat{S}[\vec{u}]+\frac{1}{|\aset{V}|}\int_{\aset{V}\setminus\cup\aset{E}_k}-p'\mat{I}+\nabla_{\vec{\xi}}\vec{u}'+\nabla_{\vec{\xi}}\vec{u}'^T\diffd \vec{\xi}+\frac{1}{|\aset{V}|}\int_{\cup\aset{E}_k}\mat{T}'\diffd\vec{\xi},
\end{align}
where the stress fluctuations outside the particle are treated as Newtonian and inside the particles with respect to an appropriate extension. We particularly model the velocity and stress fluctuations by means of the local disturbance fields given in Lemma~\ref{lem_asymAna_junk} for $\vec{u}'$ and in Lemma~\ref{lem_defiStokesParticle_asymptoticModel} for $\mat{T}'$. With (A2) we may restrict to the influence of the local fields generated by the $k$th particle on $\aset{E}_k$ and neglect particle interactions. Using the two identities
\begin{align*}
 \sum_j\int_{\aset{E}_\ell}\nabla_{\vec{\xi}}\cdot(T'_{ij}\xi_k\vec{e}_j)\diffd\vec{\xi}&=\sum_j\int_{\aset{E}_\ell}\left(\partial_{\xi_j}T'_{ij}\xi_k+T'_{ij}\delta_{jk}\right)\diffd\vec{\xi}=\int_{\aset{E}_\ell}(\nabla_{\vec{\xi}}\cdot \mat{T}'^T)_i\xi_k\diffd\vec{\xi}+\int_{\aset{E}_\ell}T'_{ik}\diffd\vec{\xi},\\
 0=\E[\partial_{x_j}u'_i]^{\aset{V}} &=\frac{1}{|\aset{V}|}\int_{\aset{V}\setminus\cup\aset{E}_\ell}\partial_{\xi_j}u'_i\diffd\vec{\xi}+\frac{1}{|\aset{V}|}\sum_\ell\int_{\aset{E}_\ell}\nabla_{\vec{\xi}}\cdot(u'_i\vec{e}_j)\diffd\vec{\xi},
\end{align*}
and the divergence theorem on \eqref{eq_Sigmas} results in
\begin{align*}
 \E[\mat{\Sigma}^s]^{\aset{V}} &=  \mat{S}[\vec{u}]-\frac{1}{|\aset{V}|}\int_{\aset{V}\setminus\cup\aset{E}_k}\hspace{-0.7cm}p'\mat{I}\diffd\vec{\xi} +\frac{1}{|\aset{V}|}\sum_k \left(\int_{\partial\aset{E}_k}\hspace{-0.4cm}(\mat{T}'\cdot\vec{n})\otimes\vec{\xi} - \vec{u}'\otimes\vec{n}-\vec{n}\otimes\vec{u}'\diffd s(\vec{\xi})-\int_{\aset{E}_k}\hspace{-0.2cm}(\nabla_{\vec{\xi}}\cdot\mat{T}'^T)\otimes\vec{\xi}\diffd\vec{\xi}\right).
\end{align*}
Hence, we obtain for the averaged surface stresses
\begin{subequations}
\begin{align}
 \E[\mat{\Sigma}^s]^{\aset{V}} &= \mat{S}[\vec{u}]+\mat{S}^p,\\\label{eq_Sp}
\mat{S}^p&= \frac{1}{|\aset{V}|}\sum_k\Big(\int_{\partial\aset{E}_k}(\mat{S}[\mat{R}^k\cdot\vec{u}_{loc}^{k}]\cdot\vec{n})\otimes\vec{\xi} - \mat{R}^k\cdot\vec{u}_{loc}^{k}\otimes\vec{n}-\vec{n}\otimes\vec{u}_{loc}^{k}\cdot\mat{R}^{kT}\diffd s(\vec{\xi})\\
&\qquad \qquad \quad -\int_{\aset{E}_k}\left(\nabla_{\vec{\xi}}\cdot\left(\mat{R}^k\cdot\mat{T}_{loc}^k\cdot\mat{R}^{kT}\right)^T\right)\otimes\vec{\xi}\diffd\vec{\xi}\Big)\nonumber
 \end{align}
 \end{subequations}
where $\mat{S}^p$ denotes the particle-induced stress tensor. The arising pressure term is incorporated here in the Newtonian stresses (as Lagrange multiplier to the incompressibility constraint). To express the integrals in $\mat{S}^p$ \eqref{eq_Sp} with respect to the particle reference state we use the identities
\begin{align*}
 \mat{S}[\mat{R}\cdot\vec{u}_{loc}]&=\sum_{i=1}^\infty -\epsilon^{i-1}p_{loc,i}\mat{I}+\epsilon^i\partial_{\vec{x}}(\mat{R}\cdot\vec{u}_{loc,i})+\epsilon^i\partial_{\vec{x}}(\mat{R}\cdot\vec{u}_{loc,i})^T \\
 &= \sum_{i=1}^\infty\epsilon^{i-1}\mat{R}\cdot (-p_{loc,i}\mat{I}+\partial_{\vec{y}}\vec{u}_{loc,i}+\partial_{\vec{y}}\vec{u}_{loc,i}^T)\cdot\mat{R}^T = \mat{R}\cdot\sum_{i=1}^\infty\epsilon^{i-1}\mat{S}[\vec{u}_{loc,i}] \cdot\mat{R}^T,\\
 \nabla_{\vec{x}}\cdot\left(\mat{R}\cdot\mat{T}_{loc}\cdot\mat{R}^{T}\right)^T &= \mat{R}\cdot\sum_{i=1}^\infty \epsilon^{i-2}\nabla_{\vec{y}}\cdot\mat{T}_{loc,i}^T,
 \end{align*}
 and
 \begin{align*}
 \int_{\partial\aset{E}_k}f(\vec{y}(\vec{x},t,\epsilon))n_j\diffd s(\vec{x}) &= \int_{\aset{E}_k}\partial_{x_j}f\diffd\vec{x}=\epsilon^{-1}\sum_\ell\int_{\aset{E}_k}\partial_{y_\ell}fR^T_{\ell j}\diffd\vec{x}\\
 &=\epsilon^{2}\sum_\ell R^T_{\ell j}\int_{\aset{E}}\partial_{y_\ell}f(\vec{y})\diffd\vec{y}=\epsilon^{2}\int_{\partial\aset{E}}f(\vec{y})(\mat{R}\cdot\vec{n})_j\diffd s(\vec{y})
\end{align*}
for any scalar-valued smooth function $f$. This implies
\begin{align*}
 \mat{S}^p&=\frac{1}{|\aset{V}|}\sum_k\sum_{i=1}^\infty\epsilon^{i+1}\mat{R}^k\cdot\int_{\partial\aset{E}}\mat{S}[\vec{u}_{loc,i}^{k}]\cdot\vec{n}\otimes(\epsilon\mat{R}^k\cdot\vec{y}+\vec{c}^k)\diffd s\\
&\quad -\epsilon^{i+2}\mat{R}^k\cdot\int_{\partial\aset{E}}\vec{u}_{loc,i}^k\otimes\vec{n}+\vec{n}\otimes\vec{u}_{loc,i}^k\diffd s(\vec{y})\cdot\mat{R}^{kT}
-\epsilon^{i+1}\mat{R}^k\cdot\int_{\aset{E}}\nabla_{\vec{y}}\cdot\mat{T}_{loc,i}^{kT}\otimes(\epsilon\mat{R}^k\cdot\vec{y}+\vec{c}^k)\diffd \vec{y}.
\end{align*}
Using Lemma~\ref{lem_defiStokesParticle_asymptoticModel}, the terms associated with the center of mass of each particle $\vec{c}^k$ cancel each other out and, since $\hat{\vec{f}}_1$ is independent of $\vec{y}$ and $\int_{\aset{E}}\vec{y}\diffd\vec{y}=\vec{0}$ by definition of center of mass in reference state, we get
\begin{align*}
 \mat{S}^p=\frac{1}{|\aset{V}|}\sum_k\epsilon^{3}\mat{R}^k\cdot\int_{\partial\aset{E}}\mat{S}[\vec{u}_{loc,1}^{k}]\cdot\vec{n}\otimes(\mat{R}^k\cdot\vec{y})-\vec{u}_{loc,1}^k\otimes\vec{n}-\vec{n}\otimes\vec{u}_{loc,1}^k\diffd s(\vec{y})\cdot\mat{R}^{kT}+\gO(\epsilon^4).
\end{align*}

\subsubsection*{Body forces}
For the random body force we use a similar approach. We assume that the overall fluctuating force in the averaging volume $\aset{V}(\vec{x},t)$ is generated by contributions of each single particle in the corresponding particle domain 
\begin{align*}
 \vec{b}(\vec{x},t,\omega)=\sum_k\vec{b}^k(\vec{x},t,\omega)\,\mathbb{I}_{\aset{E}_k(t)}(\vec{x}),
\end{align*}
where $\mathbb{I}$ denotes the indicator function, i.e.\ $\mathbb{I}_{\aset{A}}(\vec{x})=1$ for $\vec{x}\in\aset{A}$ and zero otherwise. Additionally, since the particles do not interact with each other, the random force generated by a single particle is modeled as the divergence of the local stresses in $\aset{E}_k$. The average $\E[\vec{b}]^{\aset{V}}(\vec{x},t)$ then follows as
\begin{align*}
 \E[\vec{b}]^{\aset{V}}&=\frac{1}{|\aset{V}|}\sum_k\int_{\aset{V}}\nabla_{\vec{\xi}}\cdot\left(\mat{R}^k\cdot\mat{T}^k_{loc}\cdot\mat{R}^{kT}\right)^T\mathbb{I}_{\aset{E}_k}\diffd\vec{\xi}=\frac{1}{|\aset{V}|}\sum_k\int_{\aset{E}_k}\nabla_{\vec{\xi}}\cdot\left(\mat{R}^k\cdot\mat{T}^k_{loc}\cdot\mat{R}^{kT}\right)^T\diffd\vec{\xi}\\
 &=\frac{1}{|\aset{V}|}\sum_k\mat{R}^k\cdot\sum_{i=1}^\infty\epsilon^{i+1}\int_{\aset{E}}\nabla_{\vec{y}}\cdot\mat{T}_{loc,i}^{kT}\diffd\vec{\xi}=\frac{1}{|\aset{V}|}\sum_k\mat{R}^k\cdot\sum_{i=1}^2\epsilon^{i+1}\vec{f}_i^k+\gO(\epsilon^4).
\end{align*}
The last equality holds by Lemma~\ref{lem_defiStokesParticle_asymptoticModel} and Remark~\ref{rem_defiStokesParticle_1},~2).

\subsubsection*{Reynolds-kind stress tensor}
The last step is the treatment of the Reynolds-kind stress term appearing in~\eqref{eq_kinMod_kinmodel}. The key here is the fact that the local disturbance velocity is scaled with $\epsilon$ and decreases fast enough at distance from the particle, while in its vicinity as well as in the particle domain it is bounded: By Lemma~\ref{lem_asymAna_junk} the local fields fulfill $\vec{u}_{loc,i}\sim \gO(r^{-1})$ for $r=\|\vec{y}\|\gg0$ and, since $\|\vec{y}\|=\epsilon^{-1}\|\vec{x}-\vec{c}\|$, it holds $\vec{u}_{loc}\sim\gO(\epsilon^2)$ for $\|\vec{x}-\vec{c}\|\gg0$. By (A2) and (A5) we find around every particle in $\aset{V}$ a ball $\aset{B}_{r_k}$ of radius $r_k$, containing only the particle $\aset{E}_k$ and it holds:
\begin{align*}
 \E[\vec{u}'\otimes\vec{u}']^{\aset{V}}&=\frac{1}{|\aset{V}|}\int_{\aset{V}}\vec{u}'\otimes\vec{u}'\diffd\vec{\xi}=\frac{1}{|\aset{V}|}\sum_{k=1}\mat{R}^k\cdot\int_{\aset{V}}\vec{u}_{loc}^k\otimes\vec{u}_{loc}^{k}\diffd\vec{\xi}\cdot\mat{R}^{kT}\\
 &=\frac{1}{|\aset{V}|}\sum_{k=1}\mat{R}^k\cdot\int_{\aset{B}_{r_k}}\vec{u}_{loc}^k\otimes\vec{u}_{loc}^{k}\diffd\vec{\xi}\cdot\mat{R}^{kT}+\gO(\epsilon^4)\\
 &=\frac{\epsilon^{5}}{|\aset{V}|}\sum_{k=1}\mat{R}^k\cdot\left(\int_{\tilde{\aset{B}}_{\tilde{r}_k}\setminus\aset{E}}\vec{u}_{loc,1}^k\otimes\vec{u}_{loc,1}^{k}\,\diffd\vec{y}+\int_{\aset{E}}\vec{u}_{loc,1}^k\otimes\vec{u}_{loc,1}^{k}\,\diffd\vec{y}\right)\cdot\mat{R}^{kT}+\gO(\epsilon^4),
\end{align*}
where $\tilde{\aset{B}}_{\tilde{r}_k}$ denotes the ball centered at the origin with radius $\tilde{r}_k$. The first integral in the last expression is bounded by Lemma~\ref{lem_asymAna_junk} and the last by Remark~\ref{rem_defiStokesParticle_1}. Altogether we get $\E[\vec{u}'\otimes\vec{u}']^{\aset{V}}\sim\gO(\epsilon^4)$.

\subsection{Asymptotical suspension model}

By means of the macroscopic stress and force descriptions we formulate an asymptotic suspension model to \eqref{eq_kinMod_kinmodel} that is consistent to the one-particle model and valid up to order $\gO(\epsilon^4)$ in the size parameter $\epsilon$.

\begin{theorem}[Asymptotical model of a suspension with weakly inertial tracer particles]\label{thm_kinMod_centralthm}
Let Assumption~\ref{assum_kinMod_assumptions1} be fulfilled, and let the local behavior of any particle in the suspension be determined by Lemma~\ref{lem_asymAna_junk} and Lemma~\ref{lem_defiStokesParticle_asymptoticModel}. Then the suspension is macroscopically described up to an error of $\gO(\epsilon^4)$ by
\begin{subequations}\label{eq_kinMod_thmEqs}
\begin{align}
 \Re\left(\partial_t\vec{u}+(\vec{u}\cdot\nabla_{\vec{x}})\vec{u}\right)&=\nabla_{\vec{x}}\cdot\left(\mat{S}[\vec{u}]+\mat{\Sigma}^p\right)^T+\vec{b}^p+\Re\Fr^{-2}\vec{e}_g,\quad \nabla_{\vec{x}}\cdot\vec{u}=0&&\text{in }\Omega\times\R^+,\label{eq_kinMod_kinmodel_full}
\end{align}
supplemented with appropriate initial and boundary conditions. The respective particle-induced stress and force are
\begin{align}
 \mat{\Sigma}^p(\vec{x},t) &= \frac{1}{|\aset{V}(\vec{x},t)|}\sum_{k=1}^{N(\vec{x},t)}\epsilon^3\mat{R}_0^k(t)\cdot\int_{\partial\aset{E}}\mat{S}[\vec{u}_{loc,1}^{k}](\vec{y},t)\cdot\vec{n}\otimes\vec{y}\label{eq_kinMod_regularStress}\\
 &\hspace*{25ex}-\vec{h}_{1}^k(\vec{y},t)\otimes\vec{n}-\vec{n}\otimes\vec{h}_{1}^k(\vec{y},t)\diffd s(\vec{y})\cdot\mat{R}_0^{kT}(t),\nonumber\\
 \vec{b}^p(\vec{x},t) &= \frac{1}{|\aset{V}(\vec{x},t)|}\sum_{k=1}^{N(\vec{x},t)}\left(\epsilon^2\mat{R}^k_0(t)\cdot\vec{f}^k_1(t)+\epsilon^3\left(\mat{R}^k_0(t)\cdot\vec{f}^k_2(t)+\mat{R}^k_1(t)\cdot\vec{f}^k_1(t)\right)\right),\label{eq_kinMod_singularStress}
\end{align}
\end{subequations}
with $N(\vec{x},t)$ being the number of particles in the averaging volume $\aset{V}(\vec{x},t)$ and with index $^k$ marking the asymptotic coefficients of the $k$th particle. 
\end{theorem}

The form of the particle-stress $\mat{\Sigma}^p$ in~\eqref{eq_kinMod_regularStress} is known from the work by Batchelor~\cite{Batchelor1970}, while our asymptotical approach for inertial particles (Section~\ref{subsec_mathMod_inertia}) give rise to an additional body force $\vec{b}^p$ that is generated by small deviations of the particles' center of mass from the streamlines of the surrounding fluid. Being $\gO(\epsilon^2)$, this extra-force dominates the particle contribution to the momentum of the fluid in the case of heavy tracer particles. One should further notice that the suspension behaves like a Newtonian fluid up to an error of $\gO(\epsilon^2)$ in consistency to the underlying Assumption~\ref{assum_kinMod_assumptions1}, (A7).

%%%%%%%%%%%%%%%%%%%%%%%%%%%% section 3 %%%%%%%%%%%%%%%%%%%%%%%%%%%%%%%%%%%%%%%%%
%%%%%%%%%%%%%%%%%%%%%%%%%%%%%%%%%%%%%%%%%%%%%%%%%%%%%%%%%%%%%%%%%%%%%%%%%%%%%%%%

\section{Special case of a suspension with ellipsoidal particles}\label{sec_4_specCase} %prolate ellipsoid of revolution
The suspension of particles that have the same density as the surrounding viscous carrier fluid and the shape of prolate ellipsoids is often treated in literature; the structure of the particle-induced stresses is well-known in this case, see \cite{Batchelor1970} and \cite{Leal1973}. In this section we illustrate the inertial particles' effects by comparing our asymptotical suspension model of Theorem~\ref{thm_kinMod_centralthm} with those classical results. The determination of the respective analytical forms for $\mat{\Sigma}^p$ and $\vec{b}^p$ requires knowledge of the Newtonian stresses of the local disturbance velocity $\mat{S}[\vec{u}_{loc,1}]$ at $\partial\aset{E}$, an explicit form of the Dirichlet conditions for $\vec{u}_{loc,1}$ and also of the differential equation for $\vec{c}_1$, which is hidden in the solvability conditions in Lemma~\ref{lem_asymAna_solvCond}. Therefore, we first study the solvability conditions, deriving the ODE for $\vec{c}_1$ and an expression for $\vec{\omega}_0$, also characterizing the Dirichlet conditions for $\vec{u}_{loc,1}$ and then use the corresponding local stresses to evaluate the integrals in \eqref{eq_kinMod_thmEqs}. The necessary calculations are quite technical and lengthy as they are strongly related to the geometrical properties of the ellipsoids. For the sake of completeness they are summerized in the appendix.

\subsection{General properties}

We start with introducing the quantities that characterize the geometry of arbitrary ellipsoids and discuss then the implications arising for the asymptotical framework.

An ellipsoid is given by $\aset{E}=\mat{D}\aset{B}_1$, where $\aset{B}_1$ denotes the unit ball in $\R^3$ and $\mat{D}=\diag(d_1,d_2,d_3)$  is the diagonal matrix with the lengths of the semi axes $d_i>0$, $i=1,2,3$.  The surface moments appearing in the solvability conditions \eqref{eq_asymAn_solvCondAll} in Lemma~\ref{lem_asymAna_solvCond} can be provided as
\begin{align*}
\vec{s}_q&=\vec{0},\qquad \hspace*{0.7ex} \vec{s}_{q+3}=\zeta_{q+3}|\aset{E}|(\mat{D}^2-\tr(\mat{D}^2)\mat{I})\vec{e}_{q}, \qquad \vec{t}_q= 3\zeta_q|\aset{E}|\vec{e}_q, \qquad\hspace*{6ex} \vec{t}_{q+3}=\vec{0},\\
  \mat{V}_q &= \vec{0},  \qquad \mat{V}_{q+3}=\zeta_{q+3}|\aset{E}|\mat{D}^2\cdot B(\vec{e}_{q}), \qquad\hspace*{5.4ex} W_q=\zeta_q|\aset{E}|\mat{D}^2\otimes\vec{e}_q, \qquad W_{q+3}=0,
 \end{align*}
with the geometry dependent constants $\zeta_{q}$, $q=1,2,3$. Appendix~\ref{subsec_appAnSt_detSurfMom} is dedicated to their derivation, the explicit expressions for $\zeta_{q}$ are stated at its end. From~\eqref{eq_asymAn_solvCondAll} we see that the structure of the surface moments implies a decoupling of the conditions for the linear and angular velocities, hence we get
 \begin{subequations}
 \begin{align}
  \mat{R}_0^T\cdot\vec{v}_1 &= \mat{R}_0^T\cdot\partial_{\vec{x}}\vec{u}_0\cdot\vec{c}_1 +(3|\aset{E}|\diag(\zeta_1,\zeta_2,\zeta_3))^{-1}\cdot\vec{f}_1,\label{eq_specCase_linode1}\\
   \mat{R}_0^T\cdot\vec{\omega}_0&=0.5\mat{R}_0^T\cdot\nabla_{\vec{x}}\times\vec{u}_0+|\aset{E}|^{-1}(\tr{\mat{D}^2}\mat{I}-\mat{D}^2)^{-1}\cdot\diag(\zeta_4,\zeta_5,\zeta_6)^{-1}\cdot\vec{g}_1\label{eq_specCase_rotode1}\\
 &\quad+(\tr{\mat{D}^2}\mat{I}-\mat{D}^2)^{-1}\cdot\diag(d_2^2-d_3^2,d_3^2-d_1^2,d_1^2-d_2^2)\cdot\begin{pmatrix}(\mat{R}_0^T\cdot\mat{E}[\vec{u}_0]\cdot\mat{R}_0)_{32}\\(\mat{R}_0^T\cdot\mat{E}[\vec{u}_0]\cdot\mat{R}_0)_{13}\\(\mat{R}_0^T\cdot\mat{E}[\vec{u}_0]\cdot\mat{R}_0)_{21}\end{pmatrix},\nonumber
 \end{align}
\end{subequations}
with $\mat{E}[\vec{u}_0]=0.5(\nabla_{\vec{x}}\vec{u}_0+\nabla_{\vec{x}}\vec{u}_0^T)$. In our set up, $\vec{g}_1\equiv\vec{0}$ always holds (cf.\ Abbreviation~\ref{abb_asymAn_1}). Consequently, the Dirichlet condition \eqref{eq_asymAn_dirichlet} is presented by
\begin{align*}
 \vec{h}_1 &= \mat{R}_0^T\cdot\left(\vec{v}_1-\partial_{\vec{x}}\vec{u}_0\cdot\vec{c}_1\right)+\mat{R}_0^T\cdot\left( B(\vec{\omega}_0)-\partial_{\vec{x}}\vec{u}_0\right)\cdot\mat{R}_0\cdot\vec{y}=\vec{h}_1^{const}+\mat{A}_1\cdot\vec{y}.
\end{align*}
where the matrix-valued function $\mat{A}_1=\mat{A}_1(t)$ is independent of $\vec{y}$ and the time-dependent vector $\vec{h}_1^{const}(t)$ becomes
\begin{align*}
 \vec{h}_1^{const} &= (3|\aset{E}|\diag(\zeta_q,q=1,2,3))^{-1}\cdot\vec{f}_1,
\end{align*}
because $\vec{v}_1-\partial_{\vec{x}}\vec{u}_0\cdot\vec{c}_1$ is exclusively determined by the source term in \eqref{eq_specCase_linode1}.
The last ingredient for the suspension model are the Newtonian stresses of $\vec{u}_{loc,1}$ at $\partial\aset{E}$. Since the Stokes problems in Lemma~\ref{lem_asymAna_junk} are linear, the local velocity field $\vec{u}_{loc,1}$ is given by the linear combination of the Oberbeck $\vec{u}_{loc,1,Ob}$ and Jeffery solutions $\vec{u}_{loc,1,Je}$ (see Lemma~\ref{lem_appAnSt_OberJeff} stated in Appendix~\ref{subsec_appAnSt_ObnJeff}). This yields the following stress terms that can be computed by means of the techniques provided in Appendix~\ref{subsec_appAnSt_detSurfMom},
\begin{subequations}\label{eq_specCase_stresses}
\begin{align}
 &\int_{\partial\aset{E}}\mat{S}[\vec{u}_{loc,1,Ob}]\cdot\vec{n}\diffd s = \vec{f}_1,  \hspace*{2cm}  \int_{\partial\aset{E}}\mat{S}[\vec{u}_{loc,1,Je}]\cdot\vec{n}\diffd s = \vec{0},\\
  &\int_{\partial\aset{E}}\mat{S}[\vec{u}_{loc,1,Ob}]\cdot\vec{n}\otimes\vec{y}\diffd s = \mat{0},
 \qquad \qquad \int_{\partial\aset{E}}\vec{h}_1^{const}\otimes\vec{n}+\vec{n}\otimes\vec{h}_1^{const}\diffd s = \mat{0},\\
 &\int_{\partial\aset{E}}\mat{S}[\vec{u}_{loc,1,Je}]\cdot\vec{n}\otimes\vec{y}-\mat{A}_1\cdot\vec{y}\otimes\vec{n}-\vec{n}\otimes\mat{A}_1\cdot\vec{y}\diffd s = |\aset{E}|\left(\frac8\delta \mat{M}-4(\mat{M}:\diag(\alpha_1^o,\alpha_2^o,\alpha_3^o))\mat{I}\right),
\end{align}
\end{subequations}
where $\delta=d_1d_2d_3$. For the definition of the geometry dependent scalars $\alpha_i^o$ and matrix $\mat{M}$ we refer to Appendix~\ref{subsec_appAnSt_ObnJeff}.

\subsection{Suspension model of weakly inertial ellipsoidal particles}

We combine the results from the one-particle asymptotics and deduce the macroscopic suspension description for arbitrarily shaped tracer ellipsoids (TE) that we even specify for prolate ellipsoids in Corollary~\ref{lem_specCase_ellipsoidalSusp}.

\subsubsection*{Arbitrarily shaped ellipsoids}
The particle-induced stress tensor for arbitrarily shaped ellipsoids becomes by means of  \eqref{eq_kinMod_regularStress} and \eqref{eq_specCase_stresses} 
\begin{align}
 \mat{\Sigma}^p &= \epsilon^3\frac{|\aset{E}|}{|\aset{V}|}\sum_{k=1}^{N}\left(\frac8\delta \mat{R}_0^k\cdot\mat{M}^k\cdot\mat{R}_0^{kT}-4(\mat{M}^k:\diag(\alpha_1^o,\alpha_2^o,\alpha_2^o))\mat{I}\right).\label{eq_specCase_stressArbitrary}
\end{align}
The force term $\vec{b}^p$~\eqref{eq_kinMod_singularStress} depends via the functions $\vec{f}_i^k$ (Abbreviation~\ref{abb_asymAn_1}) on the choice of the mass function~\eqref{eq_mathMod_massfcn}. With respect to the three different inertial types of tracer ellipsoids we obtain
\begin{align*}
 \vec{b}^p &= \epsilon^2 \frac{1}{|\aset{V}|}\sum_{k=1}^{N}\left(m\Re\cdot\left(\frac{\diffd}{\diffd t}\vec{v}_0^k-\frac{1}{\Fr^2}\vec{e}_g\right)+\epsilon\left(-|\aset{E}|\nabla\cdot\mat{S}[\vec{u}]^T+m\Re\frac{\diffd}{\diffd t}\vec{v}_1^k\right)\right) &&\\
 &=\epsilon^2 \phi\frac{m}{|\aset{E}|}\nabla\cdot\mat{S}[\vec{u}]^T+\epsilon^3\phi\left(\frac{m}{|\aset{E}|}\Re\E\left[\frac{\diffd}{\diffd t}\vec{v}_1\right]-\nabla\cdot\mat{S}[\vec{u}]^T\right)+\gO(\epsilon^4), &&\text{heavy TE},
 \end{align*}
 \begin{align*}
 \vec{b}^p &= \epsilon^3\frac{1}{|\aset{V}|}\sum_{k=1}^{N}\left(-|\aset{E}|\nabla\cdot\mat{S}[\vec{u}]^T+m\Re\left(\frac{\diffd}{\diffd t}\vec{v}_0^k-\frac{1}{\Fr^2}\vec{e}_g\right)\right) &&\\
 &=\epsilon^3\phi\left(\frac{m}{|\aset{E}|}-1\right)\nabla\cdot\mat{S}[\vec{u}]^T+\gO(\epsilon^5), &&\text{normal TE},\\
 \vec{b}^p &= -\epsilon^3\frac{|\aset{E}|}{|\aset{V}|}\sum_{k=1}^{N}\nabla\cdot\mat{S}[\vec{u}]^T=-\epsilon^3\phi\nabla\cdot\mat{S}[\vec{u}]^T, &&\text{light weighted TE},
\end{align*}
where $\phi=|\aset{E}|N/|\aset{V}|$ denotes the volume fraction. We use here the asymptotic approximation $\diffd\vec{v}_0/\diffd t=\diffd \vec{u}/\diffd t=\Re^{-1}\nabla\cdot\mat{S}[\vec{u}]^T+\Fr^{-2}\vec{e}_g+\gO(\epsilon^2)$ (cf.\ Theorem~\ref{thm_kinMod_centralthm}) and Assumption~\ref{assum_mathMod_assumptions1}, (A8) that allows us to  evaluate the gradient of the velocity at $\vec{x}$ instead of the center of mass of the corresponding particle. 

\begin{remark}[Macroscopic definition of the local quantities]
In the derivation of $\vec{b}^p$ we used the law of large numbers to replace the discrete arithmetic mean of $N$ particles with the expectation $\E[.]$, which is in consistency with Assumption~\ref{assum_kinMod_assumptions1}. This transformation was especially done to achieve a description that is easily comparable with the results presented in literature, see e.g.\ \cite{Batchelor1970,Leal1973,Phan-Thien1991}. However, it involves a small technical issue: While the ensemble average needs only the well-defined velocity fields of each particle $\vec{v}_1^k$ in $\aset{V}(\vec{x},t)$, the expression $\E[\diffd \vec{v}_1/\diffd t](\vec{x},t)$ presupposes the definition of an underlying random velocity field $\vec{v}_1:\Omega\times\R^+_0\times\aset{W}\to\R^3$ whose characteristic for the $k$th particle $\vec{v}_1^k(t,\omega) = \vec{v}_1(\vec{c}_0^k(t,\omega),t,\omega)$ solves almost surely the corresponding ODE~\eqref{eq_specCase_linode1}. Having the meaning of the averages of the particle related quantities in mind we stick to the presented notation for reasons of readability. 
\end{remark}

\subsubsection*{Prolate ellipsoids}
For a prolate ellipsoid the lengths of the semi axes satisfy $d_1>d_2=d_3>0$. It is convenient to introduce the aspect ratio $a_r=d_1/d_2>1$ and the parameter $\nu=(a_r^2-1)(a_r^2+1)^{-1}$. Its rotational behavior can be expressed in terms of the main director $\vec{p}_1$, $(\mat{R}_0)_{ij}=\vec{e}_i\cdot\vec{p}_j$ as the angular velocity $\vec{\omega}_0$~\eqref{eq_specCase_rotode1} becomes
\begin{subequations}\label{eq_Jeff}
\begin{align}\nonumber
 \vec{\omega}_0 &= 0.5\nabla_{\vec{x}}\times\vec{u}_0 -\nu (\vec{p}_1\cdot\mat{E}[\vec{u}_0]\cdot\vec{p}_3)\vec{p}_2+\nu (\vec{p}_2\cdot\mat{E}[\vec{u}_0]\cdot\vec{p}_3)\vec{p}_3 \\
 &= 0.5\nabla_{\vec{x}}\times\vec{u}_0+\nu\vec{p}_1\times\mat{E}[\vec{u}_0]\cdot\vec{p}_1,
\end{align}
due to the orthonormality of $\{\vec{p}_1,\vec{p}_2,\vec{p_3}\}$. In combination with 
\begin{align}
 \frac{\diffd}{\diffd t}\vec{p}_1 = \vec{\omega}_0\times\vec{p}_1,
\end{align}
\end{subequations}
\eqref{eq_Jeff} is often called Jeffery's equation in remembrance of \cite{Jeffery1922}. In the asymptotical context the equation was derived for normal tracer ellipsoids in the work of Junk \& Illner \cite{Junk2007}. As for the particle-stresses $\mat{\Sigma}^p$ \eqref{eq_specCase_stressArbitrary}, it can be shown by a technical but straight forward calculation \cite{Vibe2014} that
\begin{align*}
 \frac4\delta\mat{R}_0\cdot\mat{M}\cdot\mat{R}_0^T = a_1\vec{p}_1\otimes\vec{p}_1\otimes\vec{p}_1\otimes\vec{p}_1:\mat{E}[\vec{u}_0]+a_2(\vec{p}_1\otimes\vec{p}_1\cdot\mat{E}[\vec{u}_0]+\mat{E}[\vec{u}_0]\cdot\vec{p}_1\otimes\vec{p}_1)+a_3\mat{E}[\vec{u}_0],%\label{eq_specCase_stressClosedForm}
\end{align*}
with $\delta=d_1d_2^2$ and the geometry dependent $\mat{M}$ and $a_i$ as given in Abbreviation~\ref{abb_specCase_BigAbb} (cf.\ Appendix~\ref{subsec_appAnSt_ObnJeff}). The expression holds for every particle $k$. Using Assumption~\ref{assum_mathMod_assumptions1} and shifting the isotropic part $(\mat{M}^k:\diag(\alpha_1^o,\alpha_2^o,\alpha_2^o))\mat{I}$ into the pressure of the Newtonian stresses, we can state the particle-stresses in the well-known form with volume fraction $\phi$ and deformation gradient tensor $\mat{E[\vec{u}]}$ (see e.g.\ \cite{Leal1973}),
\begin{align}
 \mat{\Sigma}^p &=\epsilon^32\phi \widetilde{\mat{\Sigma}}^p,\nonumber\\
 \widetilde{\mat{\Sigma}}^p&= a_1\E[\vec{p}_1\otimes\vec{p}_1\otimes\vec{p}_1\otimes\vec{p}_1]:\mat{E}[\vec{u}]+a_2(\E[\vec{p}_1\otimes\vec{p}_1]\cdot\mat{E}[\vec{u}]+\mat{E}[\vec{u}]\cdot\E[\vec{p}_1\otimes\vec{p}_1])+a_3\mat{E}[\vec{u}].\label{eq_specCase_defiBatchJeffStress}
\end{align}

\begin{abbreviation}[Geometrical parameters of prolate ellipsoids for particle-stresses]\label{abb_specCase_BigAbb}
The geometrical parameters $a_1, a_2, a_3$ of \eqref{eq_specCase_defiBatchJeffStress} are (cf.\ \cite{Giesekus1962, Phan-Thien1991})
\begin{align*}
 a_1&=\frac{4}{d_1d_2^2}\left(\frac{\gamma_1^o}{4\gamma_2^o\beta_1^od_2^2}-2b_1-b_2\right), \hspace*{22.6ex} a_2=\frac{4b_1}{d_1d_2^2}, \qquad  \qquad  \qquad \quad  a_3=\frac{4b_2}{d_1d_2^2},\\
 b_1&=\frac{a_r^2(2a_r^2\theta-\theta-1)}{2d_1d_2^2(a_r^4-1)\beta_2^o(d_2^2\alpha_2^o+d_1^2\alpha_1^o)}-b_2, \hspace*{16ex} b_2=\frac{1}{4\beta_1^od_2^2},\\
 \theta&= \frac{1}{2a_r(a_r^2-1)^{1/2}}\ln\frac{a_r+(a_r^2-1)^{1/2}}{a_r-(a_r^2-1)^{1/2}},\hspace*{17ex}\chi^o=\frac{2a_r\theta}{d_2},\\
  \alpha_1^o&=\frac{1}{d_1d_2^2}\frac{2}{a_r^2-1}\left(a_r^2\theta-1\right), \hspace*{29ex} \alpha_2^o=\frac{1}{d_1d_2^2}\frac{a_r^2}{a_r^2-1}\left(-\theta+1\right),\\
 \beta_1^o&=\frac{1}{d_1d_2^4}\frac{a_r^2}{4(a_r^2-1)^2}\left(3\theta+2a_r^2-5\right),\hspace*{20ex} \beta_2^o=\frac{1}{d_1d_2^4}\frac{1}{(a_r^2-1)^2}\left(-3a_r^2\theta+a_r^2+2\right),\\
 \gamma_1^o&=\frac{1}{d_1d_2^2}\frac{a_r^2}{4(a_r^2-1)^2}\left(-(4a_r^2-1)\theta+2a_r^2+1\right), \hspace*{10.2ex}\gamma_2^o=\frac{1}{d_1d_2^2}\frac{a_r^2}{(a_r^2-1)^2}\left((2a_r^2+1)\theta-3\right).
\end{align*}
In addition,
 \begin{align*}
 \mat{M}=\begin{bmatrix}A&H&G^\star\\H^\star&B&F\\G&F^\star&C\end{bmatrix}
 \end{align*}
where the coefficients are given by
\begin{align*}
&\hspace*{30ex}A=-\frac{2\gamma_1^o A_{11}-\gamma_2^o(A_{22}-A_{33})}{6(\gamma_2^{o2}+2\gamma_1^o\gamma_2^o)},\\
B&=-\frac{\gamma_2^o(2A_{22}-A_{33})-\gamma_1^oA_{11}}{6(\gamma_2^{o2}+2\gamma_1^o\gamma_2^o)},\hspace*{30.5ex}C=-\frac{\gamma_2^o(2A_{33}-A_{22})-\gamma_1^oA_{11}}{6(\gamma_2^{o2}+2\gamma_1^o\gamma_2^o)},\\
 F&=-\frac{\alpha_2^o(A_{23}+A_{32})/2-d_2^2\beta_1^o(A_{32}-A_{23})/2}{4\beta_1^od_2^2\alpha_2^o},\quad F^\star=-\frac{\alpha_2^o(A_{23}+A_{32})/2+d_2^2\beta_1^o(A_{32}-A_{23})/2}{4\beta_1^od_2^2\alpha_2^o},\\
 G&=-\frac{\alpha_2^o(A_{13}+A_{31})/2-d_1^2\beta_2^o(A_{13}-A_{31})/2}{2\beta_2^o(d_1^2\alpha_1^o+d_2^2\alpha_2^o)},\quad G^\star=-\frac{\alpha_1^o(A_{13}+A_{31})/2+d_2^2\beta_2^o(A_{13}-A_{31})/2}{2\beta_2^o(d_1^2\alpha_1^o+d_2^2\alpha_2^o)},\nonumber\\
 H&=-\frac{\alpha_1^o(A_{12}+A_{21})/2-d_2^2\beta_2^o(A_{21}-A_{12})/2}{2\beta_2^o(d_1^2\alpha_1^o+d_2^2\alpha_2^o)},\quad H^\star=-\frac{\alpha_2^o(A_{12}+A_{21})/2+d_1^2\beta_2^o(A_{21}-A_{12})/2}{2\beta_2^o(d_1^2\alpha_1^o+d_2^2\alpha_2^o)},
\end{align*}
(cf.\ Lemma~\ref{lem_appAnSt_OberJeff}) with $A_{ij}=(\mat{A}_1)_{ij}$, $\mat{A}_1=\nu B((0, -\vec{p}_1\cdot \mat{E}[\vec{u}_0]\cdot \vec{p}_3, \vec{p}_1\cdot \mat{E}[\vec{u}_0]\cdot \vec{p}_2))-\mat{R}_0^T\cdot \mat{E}[\vec{u}_0]\cdot \mat{R}_0$.
\end{abbreviation}

\begin{cor}[Asymptotical model of a suspension with prolate weakly inertial tracer ellipsoids]\label{lem_specCase_ellipsoidalSusp}
Let the assumptions of Theorem~\ref{thm_kinMod_centralthm} be valid and let the volume fraction $\phi=|\aset{E}|N/|\aset{V}|$ be independent of $\vec{x}$, then the suspension of prolate ellipsoidal particles is macroscopically described up to an error of $\mathcal{O}(\epsilon^4)$ by
 \begin{align*}
 \Re\left(\partial_t\vec{u}+(\vec{u}\cdot\nabla_{\vec{x}})\vec{u}\right)&=\eta\nabla_{\vec{x}}\cdot\mat{S}[\vec{u}]^T+\epsilon^32\phi\nabla\cdot\widetilde{\mat{\Sigma}}^{pT}+\vec{f}^{p}+\Re\Fr^{-2}\vec{e}_g,\quad 
 \nabla_{\vec{x}}\cdot\vec{u}=0, \qquad \text{in }\Omega\times\R^+,
\end{align*}
supplemented with appropriate initial and boundary conditions. It is $\widetilde{\mat{\Sigma}}^p$ according to \eqref{eq_specCase_defiBatchJeffStress} and
\begin{align*}
 \eta=\begin{cases}
      1+\epsilon^2\phi\frac{m}{|\aset{E}|}-\epsilon^3\phi,& \text{heavy TE},\\
      1+\epsilon^3\phi\left(\frac{m}{|\aset{E}|}-1\right),& \text{normal TE},\\
      1-\epsilon^3\phi, & \text{light weighted TE},
     \end{cases}
     \qquad \qquad
 \vec{f}^{p}=\begin{cases}
                \epsilon^3\phi\frac{m}{|\aset{E}|}\Re\E\left[\frac{\diffd}{\diffd t}\vec{v}_1\right],&\text{heavy TE},\\
                \vec{0}, & \text{else}.
               \end{cases}
\end{align*}
\end{cor}
\newpage
\begin{remark}[Properties of additional inertia related forces in suspension model]\hfill
\begin{itemize}
 \item[1)] Compared to the classical results in literature, e.g.~\cite{Leal1973}, our model includes a change in the overall Newtonian stress of the fluid and additionally, in the case of heavy tracer ellipsoids, an extra body force. The source of these effects can be identified in~\eqref{eq_kinMod_singularStress} to originate from the one-particle associated functions $\vec{f}_i^k$ for $i=1,2$, which in turn have the role of source terms in the ODEs for $\vec{c}_i^k$ according to the solvability conditions~\eqref{eq_asymAn_solvCondAll}, clearly seen in the case of ellipsoidal particles~\eqref{eq_specCase_linode1}. This implies that the additional macroscopical effects have their origin in the deviation of the particle center of mass from the streamlines of the undisturbed fluid, which is given by $\vec{c}^k-\vec{c}_0^k\approx\epsilon\vec{c}_1^k+\epsilon^2\vec{c}_2^k$. This dependence is also clearly observed for ellipsoidal particles. In the general case the solvability conditions for the linear and angular velocity may not decouple, but still the $\vec{f}_i$ terms are associated to the ODE-system for the correction of particle movement. Thus throughout this text we refer to these terms in Theorem~\ref{thm_kinMod_centralthm} as originating from the small disturbance.
 \item[2)] In the special case $\rho\equiv1$, we see by ~\eqref{eq_mathMod_massfcn} that $m/|\aset{E}|\equiv 1$ and the contribution due to the relative motion vanishes, similarly to the classical results in~\cite{Batchelor1970}.
\end{itemize}
\end{remark}

\begin{remark}[Splitting of the inner force]\label{rem_specCase_motSplitting}
 As mentioned in Section~\ref{subsec_kinMod_general}, the splitting approach \eqref{eq_split} for the overall inner force into a force generated by surface stresses and a volume force is only needed, since the linearity of the mean value with respect to differentiation is not generally valid for volume-based averages. Here, we want to illustrate this fact based on our results: As we see from Theorem~\ref{thm_kinMod_centralthm}, the volume average for the surface stresses yields (apart from the Newtonian stresses) only $\mat{\Sigma}^p$ that can analytically be simplified to the classical particle-stress term of Corollary~\ref{lem_specCase_ellipsoidalSusp}. This is a contribution of $\gO(\epsilon^3)$ which is dependent neither on the local coordinates $\vec{y}$ of a single particle --and thus cannot generate an effective force of $\gO(\epsilon^2)$--, nor on the mass function $\alpha_{mass}$ --and thus cannot change despite different inertia models. In contrast, the average of the volume force $\E[\vec{b}]^{\aset{V}}$ generates a contribution of $\gO(\epsilon^2)$ and changes accordingly to the choice of the inertia model.
\end{remark}

%%%%%%%%%%%%%%%%%%%%%%%%%%%% section 5 %%%%%%%%%%%%%%%%%%%%%%%%%%%%%%%%%%%%%%%%%
%%%%%%%%%%%%%%%%%%%%%%%%%%%%%%%%%%%%%%%%%%%%%%%%%%%%%%%%%%%%%%%%%%%%%%%%%%%%%%%%

\section{Conclusion}\label{sec_5_conc}
In this paper we presented a model for a suspension of small particles, which involves inertial effects. The latter were identified as the ability of particles to deviate from the streamlines of the surrounding fluid as a consequence of different scaling of the particle momentum balance. This was achieved by an appropriate scaling of the particle mass, respectively the density ratio. To keep the results of~\cite{Junk2007}, we retained the basic asymptotic approach and restricted our choice of the mass function accordingly. Having characterized the microscale behavior of the particles, we modeled the particle suspension following~\cite{Batchelor1970}, supplementing a strategy to take additional forces into account, which have their origin in the relative motion of the particles. Afterwards we gave different models for the corresponding inertial particle regimes. These models are composed of incompressible Navier-Stokes-like equations with modified, non-Newtonian stress and force. Besides the classical part in the stress, we found a modification entering the overall viscosity of the fluid. To illustrate the general results, we applied the theory to the classical example of symmetrical ellipsoidal particles.

One practical restriction of our approach is the fact that we cannot give a recipe on how to choose an inertial regime for a given particle and fluid, since our model of inertia was formulated by means of an asymptotic behavior of the particle. For a concrete situation, one has to observe the particle motion and then, based on the intensity of deviation of the particle motion from the fluid streamlines, decide which regime is applicable.

%%%%%%%%%%%%%%%%%%%%%%%%%%%% appendix %%%%%%%%%%%%%%%%%%%%%%%%%%%%%%%%%%%%%%%%%%
%%%%%%%%%%%%%%%%%%%%%%%%%%%%%%%%%%%%%%%%%%%%%%%%%%%%%%%%%%%%%%%%%%%%%%%%%%%%%%%%

\appendix
\renewcommand{\theequation}{A.\arabic{equation}}
\renewcommand{\thetable}{A.\arabic{table}}
\renewcommand{\thefigure}{A.\arabic{figure}}
\setcounter{equation}{0} \setcounter{figure}{0}
\newtheorem{lemA}{Lemma}[section]

\section{Analytical statements for ellipsoidal geometry}\label{sec_appAnSt}
In the appendix we summarize some fundamental results of different works for ellipsoidal particles, the individual results are from Oberbeck \cite{Oberbeck1876}, Edwardes \cite{Edwardes1893}, Jeffery \cite{Jeffery1922} and Junk \& Illner \cite{Junk2007}. Since we use especially the solutions of Oberbeck and Jeffery in the derivation of the suspension model for ellipsoidal particles in Section~\ref{sec_4_specCase}, we provide here the relevant arguments. In Appendix~\ref{subsec_appAnSt_ObnJeff} we introduce all relevant functions needed to formulate the solutions of Oberbeck and Jeffery in Lemma~\ref{lem_appAnSt_OberJeff}, also analyzing the decay properties of the involved quantities in Lemma~\ref{lem_asym_props_of_coefficients}. In Appendix~\ref{subsec_appAnSt_detSurfMom} we briefly present the necessary steps for the analytical computation of the surface moments arising in the solvability conditions of Lemma~\ref{lem_asymAna_solvCond} by using Lemma~\ref{lem_appAnSt_OberJeff}.

\subsection{Oberbeck and Jeffery solutions}\label{subsec_appAnSt_ObnJeff}
As stated in Section~\ref{sec_4_specCase}, an ellipsoid is a set $\aset{E}=\mat{D}\aset{B}_1$ with $\mat{D}=\diag(d_1,d_2,d_3)$, $d_i>0$ and unit ball $\aset{B}_1$ in $\R^3$. This implies $\aset{E}=\{\vec{y}\in\R^3|\,\|\mat{D}^{-1}\vec{y}\|^2<1\}$. Consider the function $\lambda:\R^3\setminus\aset{E}\to\R_0^+$ defined as
\begin{subequations}\label{eq_appAnSt_ansatzForObJeff}
\begin{align}
 \frac{y_1^2}{d_1^2+\lambda(\vec{y})}+\frac{y_2^2}{d_2^2+\lambda(\vec{y})}+\frac{y_3^2}{d_3^2+\lambda(\vec{y})}=1\label{eq_appAnSt_root}
\end{align}
whose existence and regularity are guaranteed by the implicit function theorem. The formulation of solutions for the Stokes problems in $\R^3\setminus\aset{E}$ are based on the following geometry associated functions 
\begin{align}
 \delta(\lambda)=\det(\mat{D}_\lambda)^{1/2},\qquad \chi(\lambda)=\int_\lambda^\infty\delta(s)^{-1}\,\diffd s,\qquad \alpha_j(\lambda)&=\int_\lambda^\infty(d_j^2+s)^{-1}\delta(s)^{-1}\,\diffd s,\label{eq_appAnSt_ansatzForObJeffExplicit}\\
 \beta_j(\lambda)=\int_\lambda^\infty(d_j^2+s)\delta(s)^{-3}\,\diffd s,\qquad \gamma_j(\lambda)&=\int_\lambda^\infty(d_j^2+s)s\delta(s)^{-3}\,\diffd s\nonumber
\end{align}
\end{subequations}
for $j=1,2,3$, with $\mat{D}_{\lambda}=\mat{D}_{\lambda}(\lambda)=\diag(d_i^2+\lambda,i=1,2,3)$. When $\alpha_j,\beta_j,\gamma_j$ and $\chi$ are evaluated at zero, we abbreviate the function value with the index $^o$, e.g.\ $\chi^o=\int_0^\infty\delta(s)^{-1}\,\diffd s$. Additionally, we introduce $\psi_j(\vec{y})=\beta_j(\lambda(\vec{y}))y_ky_\ell$, where $(j,k,\ell)$ is a permutation of $(1,2,3)$, and $\omega(\vec{y})=\int_{\lambda(\vec{y})}^\infty\delta(s)^{-1}f_{el}(\vec{y},s)\,\diffd s$, where $f_{el}(\vec{y},\lambda)=\vec{y}\cdot\mat{D}_\lambda^{-1}\cdot\vec{y}-1$.  In the following some derivatives of the above functions are needed.  We set $\mu(\vec{y},\lambda)=(\vec{y}\cdot\mat{D}_\lambda^{-2}\cdot\vec{y})^{-1}$, in other words $\partial_\lambda f_{el}=-\mu^{-1}$. It follows for $\vec{y}\in\R^3\setminus\aset{E}$
\begin{subequations}\label{eq_appAnSt_derivatives}
\begin{align}
 \partial_{\vec{y}}\chi(\lambda(\vec{y}))&=-2\frac{\mu}{\delta}\vec{y}\cdot\mat{D}_\lambda^{-1}, \qquad \qquad \qquad \qquad \qquad \partial_{\vec{y}}\omega=2\vec{y}\cdot\diag(\alpha_1,\alpha_2,\alpha_3),\\
 \partial_{\vec{y}\vec{y}}\chi(\lambda(\vec{y}))&=-2\frac{\mu}{\delta}\left(\mat{D}_\lambda^{-1}-2\mu\left(\mat{D}_\lambda^{-2}\cdot\vec{y}\otimes\mat{D}_\lambda^{-1}\cdot\vec{y}+\mat{D}_\lambda^{-1}\cdot\vec{y}\otimes\mat{D}_\lambda^{-2}\cdot\vec{y}\right)\right)\label{eq_appAnSt_d2chi}\\
 &\qquad-2\frac{\mu^2}{\delta}\left(-\tr(\mat{D}_\lambda^{-1})+4\mu\vec{y}\cdot\mat{D}_\lambda^{-3}\cdot\vec{y})\mat{D}_\lambda^{-1}\cdot\vec{y}\otimes\mat{D}_\lambda^{-1}\cdot\vec{y}\right),\nonumber\\
 \partial_{\vec{y}\vec{y}}\omega&=2\diag(\alpha_1,\alpha_2,\alpha_3)-4\frac{\mu}{\delta}\left(\mat{D}_\lambda^{-1}\cdot\vec{y}\otimes\mat{D}_\lambda^{-1}\cdot\vec{y}\right),\label{eq_appAnSt_d2omega}\\
 \partial_{y_iy_jy_k}\omega&=-4\frac{\mu}{\delta}\left(\frac{\delta_{jk}y_i}{(d_i^2+\lambda)(d_j^2+\lambda)}+\frac{\delta_{ij}y_k+\delta_{ik}y_j}{(d_j^2+\lambda)(d_k^2+\lambda)}\right)\\
 &\quad+8\frac{\mu^2}{\delta}\frac{y_iy_jy_k\left((d_j^2+\lambda)^{-1}+(d_k^2+\lambda)^{-1}+(d_i^2+\lambda)^{-1}+0.5\tr(\mat{D}_\lambda^{-1})-2\mu\vec{y}\cdot\mat{D}_\lambda^{-3}\cdot\vec{y}\right)}{(d_i^2+\lambda)(d_j^2+\lambda)(d_k^2+\lambda)}.\nonumber
\end{align}
Let $(i,j,k)$ be an even permutation of $(1,2,3)$, the derivatives of $\psi_i$ read as:
\begin{align}
 \partial_{y_\ell}\psi_i&=-2\frac{\mu}{\delta^3}(\mat{D}_\lambda)_{ii}(\mat{D}_\lambda^{-1}\cdot\vec{y})_\ell y_jy_k+\beta_i(\delta_{\ell j}y_k+\delta_{\ell k}y_j),\\
 \partial_{y_my_\ell}\psi_i&=\beta_i(\delta_{\ell j}\delta_{km}+\delta_{\ell k}\delta_{jm})-2y_jy_k\frac{\mu}{\delta^3}\left(2\mu(\mat{D}_\lambda^{-1}\cdot\vec{y})_m(\mat{D}_\lambda^{-1}\cdot\vec{y})_\ell+(\mat{D}_\lambda)_{ii}(\mat{D}_\lambda^{-1})_{\ell m}\right)\label{eq_appAnSt_d2psi}\\
 &\quad-2\frac{\mu}{\delta^3}(\mat{D}_\lambda)_{ii}\left((\mat{D}_\lambda^{-1}\cdot\vec{y})_m(\delta_{\ell j}y_k+\delta_{\ell k}y_j)+(\mat{D}_\lambda^{-1}\cdot\vec{y})_\ell(\delta_{jm}y_k+\delta_{km}y_j)\right)\nonumber\\
 &\quad+4y_jy_k\frac{\mu^2}{\delta^3}(\mat{D}_\lambda)_{ii}\left((\mat{D}_\lambda^{-2}\cdot\vec{y})_\ell(\mat{D}_\lambda^{-1}\cdot\vec{y})_m+(\mat{D}_\lambda^{-2}\cdot\vec{y})_m(\mat{D}_\lambda^{-1}\cdot\vec{y})_\ell\right)\nonumber\\
 &\quad-2y_jy_k\frac{\mu^2}{\delta^3}(\mat{D}_\lambda)_{ii}(4\mu\vec{y}\cdot\mat{D}_\lambda^{-3}\cdot\vec{y}-3\tr(\mat{D}_\lambda^{-1}))(\mat{D}_\lambda^{-1}\cdot\vec{y})_\ell(\mat{D}_\lambda^{-1}\cdot\vec{y})_m.\nonumber
\end{align}
\end{subequations}
From~\eqref{eq_appAnSt_d2chi}, \eqref{eq_appAnSt_d2omega} and \eqref{eq_appAnSt_d2psi} it follows
\begin{subequations}\label{eq_appAnSt_laplace}
\begin{align}
 \Delta_{\vec{y}}\chi&=-2\frac{\mu}{\delta}(\tr(\mat{D}_\lambda^{-1})-4\mu\vec{y}\cdot\mat{D}_\lambda^{-3}\cdot\vec{y}+(-\tr(\mat{D}_\lambda^{-1})+4\mu\vec{y}\cdot\mat{D}_\lambda^{-3}\cdot\vec{y}))=0,\label{eq_appAnSt_laplace_chi}\\
 \Delta_{\vec{y}}\psi_i&=-4\frac{\mu}{\delta^3}(\mat{D}_\lambda)_{ii}y_jy_k(\tr(\mat{D}_\lambda^{-1})-(\mat{D}_\lambda^{-1})_{ii})-2\frac{\mu}{\delta^3}y_jy_k(2+(\mat{D}_\lambda)_{ii}\tr(\mat{D}_\lambda^{-1}))\label{eq_appAnSt_laplace_psi}\\
 &\quad+8\frac{\mu^2}{\delta^3}(\mat{D}_\lambda)_{ii}y_jy_k\vec{y}\cdot\mat{D}_\lambda^{-3}\cdot\vec{y}-2\frac{\mu}{\delta^3}(\mat{D}_\lambda)_{ii}y_jy_k(4\mu\vec{y}\cdot\mat{D}_\lambda^{-3}\cdot\vec{y}-3\tr(\mat{D}_\lambda^{-1}))=0,\nonumber\\
 \Delta_{\vec{y}}\omega&=-4\left(-\frac12\int_{\lambda(\vec{y})}^\infty\frac{\tr(\mat{D}_s^{-1})}{\delta}\,\diffd s+\frac1\delta\right)=-4\left(\int_{\lambda(\vec{y})}^\infty\partial_s\frac1\delta\,\diffd s+\frac1\delta\right)=0\label{eq_appAnSt_laplace_omega},
\end{align}
\end{subequations}
since $\partial_\lambda\delta^{-1}=-0.5\delta^{-1}\tr(\mat{D}_\lambda^{-1})$ and $\lim_{s\to\infty}\delta^{-1}=0$. This is an important feature of the functions, which will be used in Lemma~\ref{lem_appAnSt_OberJeff}. Next, we present some asymptotic properties of these functions for $\|\vec{y}\|\to\infty$.

\begin{lemA}[Asymptotical properties]\label{lem_asym_props_of_coefficients}
 The geometry associated functions defined in~\eqref{eq_appAnSt_ansatzForObJeff} fulfill the following properties for $r=\|\vec{y}\|$, $r\to\infty$:
\begin{align*}
\lambda\sim r^2,\qquad \alpha_i\sim\frac23r^{-3},\qquad \beta_i\sim\frac25r^{-5},\qquad \gamma_i\sim\frac23r^{-3},\qquad \chi\sim2r^{-1}, \qquad
\delta \sim r^3,\qquad \mu\sim r^2,
\end{align*}
where the notation $f\sim\phi$ stands for $f=\phi+\kO(\phi)$.
\end{lemA}

\begin{proof}
 From the definition of $\lambda(\vec{y})$, it holds
 \begin{align*}
  1=\vec{y}\cdot\mat{D}_\lambda^{-1}\cdot\vec{y}\begin{cases}\leq(\min_{i}d_i^2+\lambda)^{-1}r^2\\\geq(\max_{i}d_i^2+\lambda)^{-1}r^2\end{cases},
 \end{align*}
 and thus $0\leq\min_id_i^2\leq r^2-\lambda\leq\max_id_i^2$, which implies 
 \begin{align*}
 \lim_{r\to\infty}\frac{r^2-\lambda}{r^2}=0.
 \end{align*}
 Set $\hat{d}=\max_id_i$, then the behavior of the algebraic functions $\delta$ and $\mu$ results as
\begin{align*}
 \frac{\delta-r^3}{r^3}&\leq r^{-3}\left((\hat{d}^2+\lambda)^{3/2}-r^3\right)=\left(\frac{\hat{d}^2}{r^2}+\frac{\lambda}{r^2}\right)^{3/2}-1\to0,\\
 \frac{\mu-r^2}{r^2}&\leq r^{-2}\left(\sum_iy_i^2/(\hat{d}^2+\lambda)^2\right)^{-1}-1=r^{-4}(\hat{d}^2+\lambda)^2-1=\left(\frac{\hat{d}^2}{r^2}+\frac{\lambda}{r^2}\right)^2-1\to0.
\end{align*}
For the integral functions the following statement is used: Let $g:(0,\infty)\to\R^+$ be a continuous function with $\int_0^\infty g\,\diffd s<\infty$. Define $f(t,Q)=\int_t^Qg\,\diffd s$, then for any $Q\in\R^+$ it holds
\begin{align*}
 f(\lambda(\vec{y}),Q)-f(r^2(\vec{y}),Q)=\int_{\lambda(\vec{y})}^{r^2(\vec{y})}g\,\diffd s\leq (r^2-\lambda)\max_{s\in[\lambda,r^2]}g\leq \hat{d}^2\max_{s\in[\lambda,r^2]}g.
\end{align*}
Since $Q$ was arbitrary, this is still true for $Q\to\infty$. Another statement also needed is: Let $\alpha,a,s>0$ then $1-(a/s+1)^{-\alpha}\leq\alpha a/s$. This can be directly concluded from the fact that the function $f(s)=\alpha a/s+(a/s+1)^{-\alpha}$ is strictly decreasing and $\lim_{s\to\infty}f=1$. Consequently, 
\begin{align*}
 \left|\frac{\chi-2r^{-1}}{2r^{-1}}\right|&\leq \frac{1}{2r^{-1}}\left|\int_\lambda^{r^2}\delta(s)^{-1}\,\diffd s\right|+\frac{1}{2r^{-1}}\left|\int_{r^2}^\infty \delta(s)^{-1}- s^{-3/2}\,\diffd s\right|\\
 &\leq\frac{\hat{d}^2}{2r^{-1}}\max_{s\in[\lambda,r^2]}\delta(s)^{-1}+\frac{1}{2r^{-1}}\int_{r^2}^\infty s^{-3/2}-\delta(s)^{-1}\,\diffd s\\
 &\leq\frac{\hat{d}^2\lambda^{-3/2}}{2r^{-1}}+\frac{1}{2r^{-1}}\int_{r^2}^\infty s^{-3/2}\left(1-\left(\hat{d}^2s^{-1}+1\right)^{-3/2}\right)\,\diffd s\\
 &\leq\frac{\hat{d}^2}{2}r^{-2}\left(\frac{\lambda}{r^2}\right)^{-3/2}+\frac{1}{2r^{-1}}\int_{r^2}^\infty s^{-3/2} \frac{3}{2}\hat{d}^2s^{-1}\,\diffd s=\frac{\hat{d}^2}{2}\left(r^{-2}\left(\frac{\lambda}{r^2}\right)^{-3/2}+r^{-2}\right)\to0.
\end{align*}
The same steps lead to
\begin{align*}
 \left|\frac{\alpha_i-2/3r^{-3}}{2/3r^{-3}}\right|&\leq \frac{\hat{d}^2}{2/3}r^{-2}\left(\left(\frac{\lambda}{r^2}\right)^{-5/2}+1\right)\to0,\\
 \left|\frac{\beta_i-2/5r^{-5}}{2/5r^{-5}}\right|&\leq \frac{\hat{d}^2}{2/5}r^{-2}\left(\left(\frac{\lambda}{r^2}\right)^{-7/2}+1\right)\to0,\\
 \left|\frac{\gamma_i-2/3r^{-3}}{2/3r^{-3}}\right|&\leq \frac{\hat{d}^2}{2/3}r^{-2}\left(\left(\frac{\lambda}{r^2}\right)^{-5/2}+\frac{7}{5}\right)\to0.
\end{align*}
\end{proof}

We summarize some results of \cite{Oberbeck1876, Edwardes1893, Jeffery1922} for the disturbance flow of an ellipsoid in a Stokes flow in Lemma~\ref{lem_appAnSt_OberJeff}.

\begin{lemA}[Oberbeck and Jeffery solutions]\label{lem_appAnSt_OberJeff}
Consider an ellipsoid with its geometry associated functions. Let $\vec{v}\in\R^3,\mat{A}\in\R^{3\times3}$ be a constant vector, respectively matrix, where $\tr(\mat{A})=0$. Define
 \begin{align}
  \vec{u}_{Ob}&=(\chi\mat{I}-\nabla_{\vec{y}}\chi\otimes\vec{y}+0.5\partial_{\vec{y}\vec{y}}\omega\cdot\mat{D}^2)\cdot\mat{O}\cdot\vec{v},\label{eq_appAnSt_oberbeck}\\
  \vec{u}_{Je}&=\begin{bmatrix}\partial_{y_1}\psi_1&\partial_{y_1}\psi_2&\partial_{y_1}\psi_3\\\partial_{y_2}\psi_1&\partial_{y_2}\psi_2&\partial_{y_2}\psi_3\\\partial_{y_3}\psi_1&\partial_{y_3}\psi_2&\partial_{y_3}\psi_3\end{bmatrix}\cdot\begin{pmatrix}R\\S\\T\end{pmatrix}+\nabla_{\vec{y}}\times\begin{pmatrix}U\psi_1\\V\psi_2\\W\psi_3\end{pmatrix}+\partial_{\vec{y}\vec{y}}\omega\cdot\mat{M}^T\cdot\vec{y}-\mat{M}\cdot\nabla_{\vec{y}}\omega, \label{eq_appAnSt_jeffery}
 \end{align}
as well as $p_{Ob}=-2\nabla_{\vec{y}}\chi\cdot\mat{O}\cdot\vec{v}$ and $p_{Je}=2\mat{M}:\partial_{\vec{y}\vec{y}}\omega$
with the matrices
 \begin{align*}
  \mat{O}=\diag(\chi^o+d_i^2\alpha_i^o,i=1,2,3)^{-1},\qquad \mat{M}=\begin{bmatrix}A&H&G^\star\\H^\star&B&F\\G&F^\star&C\end{bmatrix}
 \end{align*}
 and the corresponding coefficients given by
\begin{align*}
&\hspace*{30ex}A=-\frac{2\gamma_1^o A_{11}-\gamma_2^oA_{22}-\gamma_3^oA_{33}}{6(\gamma_2^o\gamma_3^o+\gamma_1^o\gamma_3^o+\gamma_1^o\gamma_2^o)},\\
B&=-\frac{2\gamma_2^oA_{22}-\gamma_3^oA_{33}-\gamma_1^oA_{11}}{6(\gamma_2^o\gamma_3^o+\gamma_1^o\gamma_3^o+\gamma_1^o\gamma_2^o)},\hspace*{30.5ex}C=-\frac{2\gamma_3^oA_{33}-\gamma_1^oA_{11}-\gamma_2^oA_{22}}{6(\gamma_2^o\gamma_3^o+\gamma_1^o\gamma_3^o+\gamma_1^o\gamma_2^o)},\\
 F&=-\frac{\alpha_2^o(A_{23}+A_{32})/2-d_3^2\beta_1^o(A_{32}-A_{23})/2}{2\beta_1^o(d_2^2\alpha_2^o+d_3^2\alpha_3^o)},\quad F^\star=-\frac{\alpha_3^o(A_{23}+A_{32})/2+d_2^2\beta_1^o(A_{32}-A_{23})/2}{2\beta_1^o(d_2^2\alpha_2^o+d_3^2\alpha_3^o)},\\
 G&=-\frac{\alpha_3^o(A_{13}+A_{31})/2-d_1^2\beta_2^o(A_{13}-A_{31})/2}{2\beta_2^o(d_3^2\alpha_3^o+d_1^2\alpha_1^o)},\quad G^\star=-\frac{\alpha_1^o(A_{13}+A_{31})/2+d_3^2\beta_2^o(A_{13}-A_{31})/2}{2\beta_2^o(d_3^2\alpha_3^o+d_1^2\alpha_1^o)},\nonumber\\
 H&=-\frac{\alpha_1^o(A_{12}+A_{21})/2-d_2^2\beta_3^o(A_{21}-A_{12})/2}{2\beta_3^o(d_1^2\alpha_1^o+d_2^2\alpha_2^o)},\quad H^\star=-\frac{\alpha_2^o(A_{12}+A_{21})/2+d_1^2\beta_3^o(A_{21}-A_{12})/2}{2\beta_3^o(d_1^2\alpha_1^o+d_2^2\alpha_2^o)}
 \end{align*}
 \begin{align*}
 R&=(A_{23}+A_{32})/(2\beta_1^o),\hspace*{11.5ex} S=(A_{13}+A_{31})/(2\beta_2^o),\hspace*{13.1ex} T=(A_{12}+A_{21})/(2\beta_3^o),\\
 U&=2d_2^2B-2d_3^2C, \hspace*{19ex} V=2d_3^2C-2d_1^2A,\hspace*{20.8ex} W=2d_1^2A-2d_2^2B.\nonumber
\end{align*}
Then $(\vec{u},p)\in\{(\vec{u}_{Ob},p_{Ob}),(\vec{u}_{Je},p_{Je})\}$ is a solution of
\begin{align*}
 \nabla_{\vec{y}}\cdot\mat{S}[\vec{u}]^T&=\vec{0},\qquad \qquad \nabla_{\vec{y}}\cdot\vec{u}=0, &&\vec{y}\in\R^3\setminus\aset{E},\\
 \vec{u}&=\vec{h},&&\vec{y}\in\partial\aset{E},\\
 \vec{u}&\to\vec{0},&&\|\vec{y}\|\to\infty,
\end{align*}
with the decay properties $\|\vec{u}\|\leq c\|\vec{y}\|^{-1}$ and $\|\nabla_{\vec{y}}\vec{u}\|,|p|\leq \hat{c}\|\vec{y}\|^{-2}$, $c,\hat c\geq 0$. The Dirichlet condition reads as $\vec{h}=\vec{v}$ for $\vec{u}=\vec{u}_{Ob}$ and $\vec{h}=\mat{A}\cdot\vec{y}$ for $\vec{u}=\vec{u}_{Je}$.
\end{lemA}

\begin{proof}
The proof of this lemma can be found in \cite{Oberbeck1876} for the translational boundary condition and in \cite{Jeffery1922} for the linear one. For the sake of completeness, we show the relevant steps here. First, using $0=\Delta_{\vec{y}}\chi=\Delta_{\vec{y}}\omega=\Delta_{\vec{y}}\psi_j$ from \eqref{eq_appAnSt_laplace} implies
\begin{align*}
 \nabla_{\vec{y}}\cdot\vec{u}_{Ob}&=\left(\nabla_{\vec{y}}\chi-\nabla_{\vec{y}}\chi-\vec{y}\Delta_{\vec{y}}\chi+0.5\nabla_{\vec{y}}\Delta_{\vec{y}}\omega\cdot\mat{D}^2\right)\cdot\mat{O}\cdot\vec{v}=0,\\
 \nabla_{\vec{y}}\cdot\vec{u}_{Je}&=\sum_j\Delta_{\vec{y}}\psi_ja_j^R+\nabla_{\vec{y}}\Delta_{\vec{y}}\omega\cdot\mat{M}^T\cdot\vec{y}+\mat{M}:\partial_{\vec{y}\vec{y}}\omega-\mat{M}:\partial_{\vec{y}\vec{y}}\omega=0.
\end{align*}
by means of the identity $\nabla\cdot\nabla\times\vec{w}=0$ for any vector $\vec{w}$. Here, $\vec{a}^R=(R,S,T)^T$. Moreover,
\begin{align*}
 \Delta_{\vec{y}}\vec{u}_{Ob}&=\left(\Delta_{\vec{y}}\chi\mat{I}-\nabla_{\vec{y}}\Delta_{\vec{y}}\chi\otimes\vec{y}-2\partial_{\vec{y}\vec{y}}\chi+0.5\partial_{\vec{y}\vec{y}}\Delta_{\vec{y}}\omega\cdot\mat{D}^2\right)\cdot\mat{O}\cdot\vec{v}=\nabla_{\vec{y}}(-2\nabla_{\vec{y}}\chi\cdot\mat{O}\cdot\vec{v}),\\
 \Delta_{\vec{y}}\vec{u}_{Je}&=\nabla_{\vec{y}}\sum_j\Delta_{\vec{y}}\psi_ja_j^R+\nabla_{\vec{y}}\times(U\Delta_{\vec{y}}\psi_1,V\Delta_{\vec{y}}\psi_2,W\Delta_{\vec{y}}\psi_3)^T+\partial_{\vec{y}\vec{y}}\Delta_{\vec{y}}\omega\cdot\mat{M}^T\cdot\vec{y}\\
 &\quad+2\nabla_{\vec{y}}\partial_{\vec{y}\vec{y}}\omega:\mat{M}^T-\mat{M}\cdot\nabla_{\vec{y}}\Delta_{\vec{y}}\omega=\nabla_{\vec{y}}(2\partial_{\vec{y}\vec{y}}\omega:\mat{M}).
\end{align*}
The matrices $\mat{O}$, $\mat{A}$ and the coefficients $R,\ldots,W$ are chosen in a way such that that the boundary conditions on $\partial\aset{E}$ hold:
\begin{align*}
 \vec{u}_{Ob} &= \Bigg(\chi^o\mat{I}+2\frac{(\vec{y}\cdot\mat{D}^{-4}\cdot\vec{y})^{-1}}{\det(\mat{D})}\mat{D}^{-2}\cdot\vec{y}\otimes\vec{y},\\
 &\qquad+0.5\left(2\diag(\alpha_i^o,i=1,2,3)-4\frac{(\vec{y}\cdot\mat{D}^{-4}\cdot\vec{y})^{-1}}{\det(\mat{D})}\mat{D}^{-2}\cdot\vec{y}\otimes\mat{D}^{-2}\cdot\vec{y}\right)\mat{D}^2\Bigg)\cdot\mat{O}\cdot\vec{v}\\
 &=\left(\chi^o\mat{I}+\diag(\alpha_i^o,i=1,2,3)\cdot\mat{D}^2\right)\cdot\diag(\chi^o+d_i^2\alpha_i^o,i=1,2,3)^{-1}\cdot\vec{v}=\vec{v},
\end{align*}
analogously for $\vec{u}_{Je}$. Last, the decreasing properties for $r=\|\vec{y}\|\to\infty$ are shown. 
%%%%%%%%%%%%%%%%%%%%%%%%%%%%%%%%%%%%%%%%%%%%%%%%%%%%%%%%%%%%%%%%%%%%%%%%%%%%%%%%%%%%%%%%%%%%%%%%%%%%%%%%%%%%%%%%%%%%%%%%%%%%%%%%% 
This follows from the definition of the velocities~\eqref{eq_appAnSt_oberbeck} and~\eqref{eq_appAnSt_jeffery} as well as the corresponding pressures, 
\begin{align*}
 \|\vec{u}_{Ob}\|&\leq c(|\chi|+\|\partial_{\vec{y}}\chi\|r+\|\partial_{\vec{y}\vec{y}}\omega\|),&& |p_{Ob}|\leq c\|\partial_{\vec{y}}\chi\|,\\
 \|\vec{u}_{Je}\|&\leq c(\max_k \|\partial_{\vec{y}}\psi_k\|+\|\partial_{\vec{y}\vec{y}}\omega\|r+\|\partial_{\vec{y}}\omega\|),&& |p_{Je}|\leq c\|\partial_{\vec{y}\vec{y}}\omega\|,\\
 \|\partial_{\vec{y}}\vec{u}_{Ob}\|&\leq c(\|\partial_{\vec{y}}\chi\|+\|\partial_{\vec{y}\vec{y}}\chi\|r+\|\partial_{\vec{y}\vec{y}\vec{y}}\omega\|), &&\\
 \|\partial_{\vec{y}}\vec{u}_{Je}\|&\leq c(\max_k\|\partial_{\vec{y}\vec{y}}\psi_k\|+\|\partial_{\vec{y}\vec{y}\vec{y}}\omega\|r+\|\partial_{\vec{y}\vec{y}}\omega\|), &&
\end{align*}
with appropriate constants $c>0$. For $r$ sufficiently big, it is known from \eqref{eq_appAnSt_derivatives}, Lemma~\ref{lem_asym_props_of_coefficients} that
\begin{align*}
 |\chi|&\leq cr^{-1}, \quad\hspace*{1.8ex} \|\partial_{\vec{y}}\chi\|\leq cr^{-2} ,\quad\hspace*{0.2ex} \|\partial_{\vec{y}\vec{y}}\chi\|\leq cr^{-3},\quad\hspace*{2ex} \|\partial_{\vec{y}}\omega\|\leq cr^{-2}\\
 \|\partial_{\vec{y}\vec{y}}\omega\|&\leq cr^{-3},\quad \|\partial_{\vec{y}\vec{y}\vec{y}}\omega\|\leq cr^{-4},\quad \|\partial_{\vec{y}}\psi_k\|\leq cr^{-4}, \quad \|\partial_{\vec{y}\vec{y}}\psi_k\|\leq cr^{-5},
\end{align*}
yielding the proposed decay properties:
\begin{align*}
 \|\vec{u}_{Ob}\|&\leq cr^{-1},\quad\hspace*{2ex} |p_{Ob}|\leq cr^{-2},\quad\hspace*{2ex} \|\partial_{\vec{y}}\vec{u}_{Ob}\|\leq cr^{-2},\\
 \|\vec{u}_{Je}\|&\leq cr^{-2},\quad |p_{Je}|\leq cr^{-3}, \quad \|\partial_{\vec{y}}\vec{u}_{Je}\|\leq cr^{-3}.
\end{align*}
%%%%%%%%%%%%%%%%%%%%%%%%%%%%%%%%%%%%%%%%%%%%%%%%%%%%%%%%%%%%%%%%%%%%%%%%%%%%%%%%%%%%%%%%%%%%%%%%%%%%%%%%%%%%%%%%%%%%%%%%%%%%%%%%% 
\end{proof}

\subsection{Determination of surface moments}\label{subsec_appAnSt_detSurfMom}
For the solvability conditions~\eqref{eq_asymAn_solvCondAll} we provide the surface moments of an ellipsoid. Computing the Newtonian stresses of Oberbeck and Jeffery solutions for the Stokes problems (Lemma~\ref{lem_appAnSt_OberJeff}) on $\partial\aset{E}$ yields
\begin{align*}
  \mat{S}[\vec{u}_{Ob}](\vec{y})&=\frac{4\mu}{\delta}\left(-(\tilde{\vec{v}}\cdot\mat{D}^{-2}\cdot\vec{y})\mat{I}-\tilde{\vec{v}}\otimes\mat{D}^{-2}\cdot\vec{y}-\mat{D}^{-2}\cdot\vec{y}\otimes\tilde{\vec{v}}+2\mu(\tilde{\vec{v}}\cdot\mat{D}^{-2}\cdot\vec{y})\mat{D}^{-2}\cdot\vec{y}\otimes\mat{D}^{-2}\cdot\vec{y}\right),\\
 \mat{S}[\vec{u}_{Je}](\vec{y})&=\mat{A}+\mat{A}^T-16\frac{\mu^2}{\delta}(\vec{y}\cdot\mat{D}^{-2}\cdot\mat{M}\cdot\mat{D}^{-2}\cdot\vec{y})\mat{D}^{-2}\cdot\vec{y}\otimes\mat{D}^{-2}\cdot\vec{y}\\
 &\qquad+\frac{8\mu}{\delta}\left(\mat{M}\cdot\mat{D}^{-2}\cdot\vec{y}\otimes\mat{D}^{-2}\cdot\vec{y}+\mat{D}^{-2}\cdot\vec{y}\otimes\mat{M}\cdot\mat{D}^{-2}\cdot\vec{y}\right)\\
 &\qquad-4\mat{M}:(\diag(\alpha_1^o,\alpha_2^o,\alpha_3^o)-2\frac{\mu}{\delta}\mat{D}^{-2}\cdot\vec{y}\otimes\mat{D}^{-2}\cdot\vec{y})\mat{I},
\end{align*}
where $\mu=(\vec{y}\cdot\mat{D}^{-4}\cdot\vec{y})^{-1}$, $\delta=\det(\mat{D})$ and $\tilde{\vec{v}}=\mat{O}\cdot\vec{v}$. For $\vec{z}\in\S_1^2$ the following identities hold:
\begin{subequations}\label{eq_appAnSt_Stresses}
\begin{align}
 \mat{S}[\vec{u}_{Ob}](\mat{D}\cdot\vec{z})\cdot\mat{D}^{-1}\cdot\vec{z}&=-\frac{4}{\delta}\tilde{\vec{v}},\label{eq_appAnSt_SOb}\\
 \mat{S}[\vec{u}_{Je}](\mat{D}\cdot\vec{z})\cdot\mat{D}^{-1}\cdot\vec{z}&=\left(\mat{A}+\mat{A}^T+\frac{8}{\delta}\mat{M}-4(\mat{M}:\diag(\alpha_i^o,i=1,2,3))\mat{I}\right)\cdot\mat{D}^{-1}\cdot\vec{z}.\nonumber
\end{align}
The last result simplifies further, if $\mat{A}$ is skew-symmetric, i.e. $\mat{A}=B(\vec{v})$ for some $\vec{v}\in\R^3$. Then,
\begin{align}
 \mat{S}[\vec{u}_{Je}](\mat{D}\cdot\vec{z})\cdot\mat{D}^{-1}\cdot\vec{z}&=-\frac{4}{\delta}B(\diag(c_i,i=1,2,3)\cdot\vec{v})\cdot\mat{D}\cdot\vec{z},\label{eq_appAnSt_SJeff}
\end{align}
\end{subequations}
with the constants $c_i=(d_j^2\alpha_j^o+d_k^2\alpha_k^o)^{-1}$, where $(i,j,k)$ is a permutation of $(1,2,3)$.

To evaluate the integrals over the surface of an ellipsoid we transform them to the unit sphere. According to~\cite{Junk2007} it holds:
\begin{align}
 \int_{\partial\aset{E}}\mat{F}(\vec{y})\cdot\vec{n}(\vec{y})\diffd s(\vec{y})&=\int_{\partial\aset{E}}\mat{F}(\vec{y})\cdot\frac{\mat{D}^{-2}\cdot\vec{y}}{\|\mat{D}^{-2}\cdot\vec{y}\|}\diffd s(\vec{y})=\int_{\S_1^2}\mat{F}(\mat{D}\cdot\vec{z})\cdot\frac{\mat{D}^{-1}\cdot\vec{z}}{\|\mat{D}^{-1}\cdot\vec{z}\|}\det(\mat{D})\|\mat{D}^{-1}\cdot\vec{z}\|\diffd s(\vec{z})\nonumber\\
 &=\delta\int_{\S_1^2}\mat{F}(\mat{D}\cdot\vec{z})\cdot\mat{D}^{-1}\cdot\vec{z}\diffd s(\vec{z}),\label{eq_appAnSt_transformToS}
\end{align}
for an integrable function $\mat{F}:\R^3\to\R^3\times\R^3$. 

Now consider the six Stokes problems in Lemma~\ref{lem_asymAna_solvCond}. According to Lemma~\ref{lem_appAnSt_OberJeff} the analytical solutions $(\vec{w}_q,p_q)$ are given by~\eqref{eq_appAnSt_oberbeck} with $\vec{v}=\vec{e}_q$ for $q=1,2,3$ and by~\eqref{eq_appAnSt_jeffery} with $\mat{A}=B(-\vec{e}_{q-3})$ for $q=4,5,6$. To calculate the surface moments $\vec{s}_q,\vec{t}_q,\mat{V}_q$ and $W_q$, we first apply the transformation to the unit sphere~\eqref{eq_appAnSt_transformToS} and afterwards use~\eqref{eq_appAnSt_SOb} or~\eqref{eq_appAnSt_SJeff} for the corresponding moments:
\begin{align*}
 \vec{s}_q&=\int_{\partial\aset{E}}B(\vec{y})\cdot\mat{S}[\vec{w}_q]\cdot\vec{n}\,\diffd s(\vec{y})\hspace*{2ex}=-4\int_{\S_1^2}B(\mat{D}\cdot\vec{z})\cdot\vec{v}_q(\vec{z})\diffd s(\vec{z}),\\
 \vec{t}_q&=\int_{\partial\aset{E}}\mat{S}[\vec{w}_q]\cdot\vec{n}\,\diffd s(\vec{y})\hspace*{8.55ex}=-4\int_{\S_1^2}\vec{v}_q(\vec{z})\diffd s(\vec{z}),\\
 (\mat{V}_q)_{ij}&=\left(\int_{\partial\aset{E}}y_i\mat{S}[\vec{w}_q]\cdot\vec{n}\,\diffd s(\vec{y})\right)_j\hspace*{2.3ex}=-4\left(\int_{\S_1^2}(\mat{D}\cdot\vec{z})_i\vec{v}_q(\vec{z})\diffd s(\vec{z})\right)_j,\\
 (W_q)_{ijk}&=\left(\int_{\partial\aset{E}}y_iy_j\mat{S}[\vec{w}_q]\cdot\vec{n}\,\diffd s(\vec{y})\right)_k=-4\left(\int_{\S_1^2}(\mat{D}\cdot\vec{z})_i(\mat{D}\cdot\vec{z})_j\cdot\vec{v}_q(\vec{z})\diffd s(\vec{z})\right)_k,
\end{align*}
with $\vec{v}_q=\mat{O}\cdot\vec{e}_q$ and $\vec{v}_{q+3}=-c_qB(\vec{e}_{q})\cdot\mat{D}\cdot\vec{z}$ for $q=1,2,3$. Since the integral of a homogeneous polynomial $p(\vec{z})$ fulfills $\int_{\S_1^2}p(\vec{z})\diffd s(\vec{z}) = 0$ if $p$ has odd degree, it follows
\begin{align*}
 \vec{s}_q&=\vec{0}, \qquad\hspace*{.5ex} \mat{V}_q = \mat{0}, \qquad \text{for $q=1,2,3$},\\
 \vec{t}_q&=\vec{0}, \qquad W_q = 0, \qquad \hspace*{.2ex} \text{for $q=4,5,6$}.
\end{align*}
The other moments follow, considering $\int_{\S_1^2}(\mat{D}\cdot\vec{z})_i(\mat{D}\cdot\vec{z})_j\diffd s=(\mat{D}^2)_{ij}4\pi/3$ and $|\S_1^2| = 4\pi$,
\begin{align*}
 (\vec{s}_{q+3})_k &= \frac{16\pi}{3}\sum_{i,j,\ell,m}(c_q\vec{e}_q)_\ell\epsilon_{jki}\epsilon_{j\ell m}(\mat{D}^2)_{im}=\frac{16\pi}{3}\sum_{i\ell m}(c_q\vec{e}_q)_\ell(\delta_{k\ell}\delta_{im}-\delta_{km}\delta_{i\ell})(\mat{D}^2)_{im}\\
 &=-c_q\frac{16\pi}{3}\left(\left(\mat{D}^2-\tr(\mat{D}^2)\mat{I}\right)\vec{e}_{q}\right)_k=\zeta_{q+3}|\aset{E}|\left(\left(\mat{D}^2-\tr(\mat{D}^2)\mat{I}\right)\vec{e}_{q}\right)_k,\\
 (\mat{V}_{q+3})_{ij} &= \frac{16\pi}{3}\sum_{k,\ell}(\mat{D}^2)_{i\ell}\epsilon_{k\ell j}(c_q\vec{e}_q)_k = -c_q\frac{16\pi}{3}(\mat{D}^2\cdot B(\vec{e}_q))_{ij} = \zeta_{q+3}|\aset{E}|(\mat{D}^2\cdot B(\vec{e}_q))_{ij},\\
 \vec{t}_q &= -16\pi\mat{O}\cdot\vec{e}_q = 3\zeta_q|\aset{E}|\vec{e}_q,\\
 (W_{q})_{ijk} &= -\frac{16\pi}{3}\mat{D}^2_{ij}(\mat{O}\cdot\vec{e}_q)_k=\zeta_q|\aset{E}|(\mat{D}^2\otimes\vec{e}_q)_{ijk},
\end{align*}
for $q=1,2,3$, where $\zeta_{q}=-4\det(\mat{D})^{-1}O_{qq}$ and $\zeta_{q+3}=-4\det(\mat{D})^{-1}c_q$, with $\mat{O}$ given in Lemma~\ref{lem_appAnSt_OberJeff} and the constants $c_i=(d_j^2\alpha_j^o+d_k^2\alpha_k^o)^{-1}$, here $(i,j,k)$ is a permutation of $(1,2,3)$.

\subsection*{Acknowledgements}
The authors acknowledge the support by the German BMBF (Project OPAL 05M13).

%%%%
\bibliographystyle{siam}
\bibliography{references}

\end{document}